\documentclass[11pt,a4paper,reqno]{amsart}
\usepackage{amsfonts}
\usepackage{amsthm}
\usepackage{amsmath}
\usepackage{amscd}
\usepackage[latin2]{inputenc}
\usepackage{t1enc}
\usepackage[mathscr]{eucal}
\usepackage{indentfirst}
\usepackage{graphicx}
\usepackage{graphics}
\numberwithin{equation}{section}
\usepackage[margin=2.9cm]{geometry}
\usepackage{epstopdf} 
\usepackage{MnSymbol}
\usepackage{float} 
\usepackage{amsmath}
\usepackage{tikz}
\usetikzlibrary{decorations.markings}
\usetikzlibrary{decorations.pathreplacing,angles,quotes}
\usepackage{url}
\usepackage{pgfplots}
\usetikzlibrary{cd}
\usetikzlibrary{calc}
\usepackage{esint}
\usepackage{empheq} 
\usepackage{enumerate}
\usepackage{graphicx}
\usepackage{subcaption} 
\usepackage{makecell}
\usepackage{longtable}
\usepackage[makeroom]{cancel}
\usepackage{mathrsfs}
\usepackage{calc}

\tikzset{->-/.style={decoration={
  markings,
  mark=at position .5 with {\arrow{>}}},postaction={decorate}}}

\theoremstyle{definition}
\newtheorem{mythm}{Theorem}[section]
\newtheorem{mydef}[mythm]{Definition}
\newtheorem{mycor}[mythm]{Corollary}
\newtheorem{mylem}[mythm]{Lemma}
\newtheorem{myprop}[mythm]{Proposition}

\newtheorem{myrem}[mythm]{Remark}

\begin{document}

\title{Ricci Flow Singularity for Triaxial Bianchi IX Metric}

\author[M.S. Johar]{M. Syafiq Johar}

\address{Modelling and Data Science Centre \\ National University of Malaysia \\ Bangi \\ Malaysia} 

\email{msyajoh@ukm.edu}

% \subjclass[2010]{Primary: 05C??. Secondary: 05C??}
% \keywords{sample paper} 

\begin{abstract} In this work, we are going to find sufficient conditions on the initial triaxial Bianchi IX metric on some $4$-dimensional manifolds foliated by homogeneous $S^3$ for a Type I singularity to occur when it is flowed under the Ricci flow. This work generalises the study on rotationally symmetric manifolds done by Angenent and Isenberg \cite{AK} as well as the work of Isenberg, Knopf, and \v Se\v sum \cite{IKS}, in which they introduced some ansatz for the problem setup.
\end{abstract}

\maketitle

\tableofcontents

\section{Introduction} The Ricci flow is an evolution equation of the metric on a Riemannian manifold $(M,g(t))$ introduced by Richard Hamilton in 1982 for the study of the famous Poincar\'e conjecture. This metric evolution equation is given by:
\begin{align} \frac{\partial}{\partial t} g(t) &= -2 \text{Ric}(g(t)), \label{ricciflow}
\\ g(0)&=g_0, \label{ricciflow2} \end{align}
where $\text{Ric}$ is the Ricci curvature of the metric $g(t)$ at time $t \geq 0$. 

This geometric flow has been important in the study of topological properties of manifolds. In particular, it was instrumental in the proof of the Poincar\'e Conjecture and the Thurston Geometrisation Conjecture by Perelman in early 2000s.

Due to the non-linear nature of the Ricci flow, we can only guarantee the short time existence of the solution to this flow from any initial data. As a result, the flow might run into a singularity at some finite time. The finite-time singularities of the Ricci flow can be classified into two types, which are called the Type I and Type II singularities, depending on the curvature asymptotics. 

A Type I singularity occurs in a Ricci flow if the singularity occurs at a finite time $T < \infty$ such that we have the following curvature asymptotics:
\begin{equation}\sup_{M \times [0,T)}(T-t)|\text{Rm}(\cdot,t)|_g < \infty, \label{type1} \end{equation}
and by dichotomy, Type II singularity occurs if we have the following curvature asymptotics:
\begin{equation}\sup_{M \times [0,T)}(T-t)|\text{Rm}(\cdot,t)|_g = \infty, \label{type2} \end{equation}
where $\text{Rm}$ is the Riemann curvature tensor and $|\cdot|_g$ is the norm on $4$-tensors induced from the metric $g$.

Studying the singularities of the Ricci flow has been a long-standing subject of research in differential geometry. There has been an interest in constructing explicit initial data which develop singularities and modelling these singularities with self-similar Ricci flow solutions, which are called the Ricci solitons. 

In this work, we are going to give some sufficient conditions for a Type I singularity to develop for the triaxial Bianchi IX metric. The theorems that we are going to prove here are:

\begin{mythm} \label{compare} Suppose that we run a Ricci flow on the manifold $(S^1 \times S^3,g_0)$ where the initial metric $g_0$ is of the form:
\begin{equation}
g_0 = \phi(z)^2 \, dz^2 +a(z)^2 \omega^1 \otimes \omega^1 + b(z)^2 \omega^2 \otimes \omega^2 + c(z)^2 \omega^3 \otimes \omega^3, \label{triax}
\end{equation}
for $z \in S^1=[0,2\pi)$ and $0 \leq a(z) \leq b(z) \leq c(z)$. Then the Ricci flow preserves the form of the metric, the ordering of the metric components, and there exist constants $\tilde{C}_0,\bar{C}_0>0$ depending on the data on the initial data such that for all $(z,t) \in S^1_T=S^1 \times [0,T)$, we have: 
\begin{align*}
|b(z,t)-c(z,t)| & \leq \tilde{C}_0 \min(b(z,t),c(z,t)),
\\ |a(z,t)-c(z,t)| & \leq \bar{C}_0 \min(a(z,t),c(z,t)).  \tag*{$\blacksquare$}
\end{align*}
\end{mythm}

\begin{mythm} \label{theorem1}
There exist open sets of warped metrics on $S^1 \times S^3$ of the form:
 $$g = \phi(z)^2 \, dz^2 +a(z)^2 \omega^1 \otimes \omega^1 + b(z)^2 \omega^2 \otimes \omega^2 + c(z)^2 \omega^3 \otimes \omega^3,$$
for $z \in S^1=[0,2\pi)$, satisfying:
\begin{enumerate}
\item $0< a \leq b \leq c < 2 a$ at $t=0$,
\item $\displaystyle \min_{s,t=0} (\text{S}) \geq 0$ at $t=0$,
\item there exists a $T < \infty$ such that $ \displaystyle \limsup_{t \rightarrow T} \max_s | \text{Ric}(s, t)|=\infty $,
\end{enumerate}
such that all solutions of the Ricci flow originating in these sets develop local neckpinch singularities at
some $T < \infty$. Each such solution has the properties that:
\begin{enumerate}
\item the ordering $a \leq b \leq c$ is preserved,
\item  the singularity is of Type I with $ | \text{Rm} | \leq \frac{C}{
 \min_s (a^2)}$ and $\displaystyle D\sqrt{T-t} \leq \min_s(a) \leq C\sqrt{T-t} $ for some constants $C,D >0$. \hfill $\blacksquare$
\end{enumerate}
\end{mythm}

\section{Previous Work}

The first non-trivial concrete example for an initial metric which leads to the Type I singularity has been constructed by Angenent and Knopf in 2004 \cite{AK}. In the paper, they constructed this metric by considering a warped product metric on $S^3$ of the form: 
\begin{equation}g_0 = ds^2+\psi(s)^2\,\hat{g}_{S^2}, \label{angknop}\end{equation}
where $\hat{g}_{S^2}$ is the canonical metric on the unit $2$-sphere and $\psi(s)$ is a function on the interval $I=\left(-\frac{\pi}{2},\frac{\pi}{2} \right)$. The function $\psi(s)$ describes the radius of the $S^2$ spheres which foliate the $S^3$ manifold over the interval $I$ such that $\psi$ vanishes at the boundaries of the interval $I$ and $\psi'(s) \rightarrow \pm 1$ as $s \rightarrow \mp \frac{\pi}{2}$. This way, they reduced the Ricci flow equation to a one-dimensional parabolic PDE, which simplifies the problem considerably.

This family of explicit examples was generalised further by Isenberg, Knopf, and \v Se\v sum \cite{IKS}. Instead of considering the warped product metric on $S^3$ of the form (\ref{angknop}), they considered a warped product metric on $S^1 \times S^3$. The base of the foliation $S^1$ simplifies the analysis in \cite{AK} considerably by doing away the technical analysis at the poles of the manifold, which are special orbits of the symmetry. Furthermore, the fibre manifold $S^3$ is parallelisable, which means that there exists a global vector field frame $\{E_1,E_2,E_3\}$ on this manifold. Therefore, a more general form for the metric on the fibre is given by:
\begin{align*}
g=a^2 \omega^1 \otimes \omega^1 + b^2 \omega^2\otimes \omega^2 +c^2 \omega^3 \otimes \omega^3,
\end{align*}
for some positive constants $a,b,$ and $c$ such that $\{\omega^1,\omega^2,\omega^3\}$ are dual frames to the vector fields $\{E_1,E_2,E_3\}$. Instead of this general form, they studied the biaxial case for simplicity; that is, they set $b \equiv c$. In short, they studied the Ricci flow of metrics on $S^1 \times S^3$ of the form:
\begin{align*}
g=\phi(z)^2 ds^2+a(z)^2 \omega^1 \otimes \omega^1 +c(z)^2 (\omega^2 \otimes \omega^2 +\omega^3 \otimes \omega^3), 
\end{align*}
where $a(z)$ and $c(z)$ are $2\pi$-periodic positive functions of $z$. Metrics of this form are also known as the biaxial Bianchi IX metric or the warped Berger metric.

Under assumptions on the initial ordering of the metric components $a$ and $c$ as well as curvature assumptions, they managed to find sufficiency conditions for a Type I singularity to occur. Furthermore, with additional conditions, they proved that the singularity resembles a shrinking cylinder after scaling, which is similar to what was done by Angenent and Knopf for the $S^3$ case.  More specifically, they proved:

\begin{mythm} \cite{IKS} \label{IKS1}
The eccentricity of every warped Berger solution of Ricci flow is uniformly bounded: there exists a constant $C_0>0$ depending only on the initial data such that the estimate:
$$|a - c| \leq C_0 \min(a,c),$$
holds pointwise for as long as the solution exists. \hfill $\blacksquare$
\end{mythm}

\begin{mythm}\label{IKS2} \cite{IKS}
There exist open sets of warped Berger metrics satisfying: 
\begin{enumerate}
\item $0<a \leq c$ at $t=0$,
\item $\displaystyle \min_{s,t=0} (\text{S}) \cdot \max_{s,t=0} (c^2)>-3$ at $t=0$,
\item there exists $T < \infty$ such that $\displaystyle \limsup_{t \rightarrow T} \max_s | \text{Ric}(s, t)| $,
\end{enumerate}

such that all solutions originating in these sets develop local neckpinch singularities at
some $T < \infty$. Each such solution has the properties that:
\begin{enumerate}
\item the ordering $a \leq c$ is preserved,
\item  the singularity is Type I with $ | \text{Ric} | \leq \frac{C}{ \min_s (a^2)}$ and $ \frac{\sqrt{T-t}}{C} \leq \displaystyle{\min_s(a)} \textstyle \leq C\sqrt{T-t} $,
\item the diameter of $(M,g)$ is bounded as $t \rightarrow T$. \hfill $\blacksquare$
\end{enumerate}
\end{mythm}

Theorems \ref{compare} and \ref{theorem1} are generalisations of Theorems \ref{IKS1} and \ref{IKS2}, which were proven in \cite{IKS}. In our case, we assume that the metric quantity $b$ is not identically equal to $c$. The main difficulty in proving our theorems is that we have to control three quantities instead of just two. In some of the equations that we are going to study, these quantities and their derivatives might be coupled. These will be dealt by utilising Young's inequality to separate the various terms. Furthermore, in order to deal with the resulting inequalities, we have to impose the extra condition such that $a \leq c < 2a$ at initial time, which was not present in Theorem \ref{IKS2}.

\section{Proof of Theorem 1.1}

From Appendix A and the work done by Isenberg, Knopf, and \v Se\v sum, we know that the Ricci flow equation from the metrics on $S^1 \times S^3$ with initial metric of the form (\ref{triax}) would preserve the symmetry and form of the metric. This is true since the metric quantities evolve according to the semilinear parabolic system of equations:
\begin{align*} 
\partial_t a &= a''+a'\left(\frac{b'}{b}+\frac{c'}{c} \right) - 2a\left(\frac{a^4-(b^2-c^2)^2}{(abc)^2}\right),
\\ \partial_t b &= b''+b'\left(\frac{a'}{a}+\frac{c'}{c} \right) - 2b\left(\frac{b^4-(a^2-c^2)^2}{(abc)^2}\right), 
\\ \partial_t c &= c''+c'\left(\frac{a'}{a}+\frac{b'}{b} \right) - 2c\left(\frac{c^4-(a^2-b^2)^2}{(abc)^2}\right), 
\end{align*}
along with the commutator relation $ [\partial_t,\partial_s]=-\left( \frac{a''}{a}+\frac{b''}{b}+\frac{c''}{c} \right) \partial_s$. We now prove that the initial ordering of the metric components is also preserved under the Ricci flow.

\begin{mylem} \cite{IKS} \label{order} If $0<a(z,0) \leq b(z,0) \leq c(z,0)$, then under the Ricci flow this ordering is preserved for all $t >0$ for which the flow exists. \hfill $\blacksquare$
\end{mylem}

\begin{proof}
Consider the parabolic PDEs satisfied by the quantities $\frac{\xi}{a}=\frac{a-b}{a}$ and $\frac{\zeta}{b}=\frac{b-c}{b}$. They evolve according to the equations (\ref{ecc1}) and (\ref{ecc3}). Thus, the maximum principles hold here where the parabolic boundary $\mathscr{P}S^1_T=S^1 \times \{0\}$ is simply the base $S^1$ at initial time. Note that $a \leq b \leq c$ initially. By applying Theorem \ref{maximumprinciple3} to the equation for $\frac{\xi}{a}$ and Theorem \ref{maximumprinciple1} to the equation for $\frac{\zeta}{b}$, we deduce that $\frac{\xi}{a} \leq 0$ and $\frac{\zeta}{b} \leq 0$ in $S^1_T$ and hence we have $0<a \leq b \leq c$ for all $(z,t) \in S^1_T$.
\end{proof}

Furthermore, by using Theorem \ref{maximumprinciple1} on equations (\ref{ecc3})-(\ref{ecc6}), in the same vein as the proof by Isenberg, Knopf, and \v Se\v sum, we can deduce the eccentricity result in Theorem \ref{compare}.

\begin{proof}[Proof of Theorem \ref{compare}]
We consider equations (\ref{ecc3}) and (\ref{ecc4}). Since $0 < a \leq c$ for all time $t \in [0,T)$, we can deduce the following inequalities: 
\begin{align*}
-4\left(\frac{1}{a^2}-\frac{1}{c^2}+\frac{b^2}{a^2c^2}+\frac{b^2+c^2-a^2}{a^2bc}\right)=-4(b^2+c^2-a^2)\left(\frac{1}{a^2c^2}+\frac{1}{a^2bc} \right) &\leq 0,
\\ -4\left(\frac{1}{a^2}-\frac{1}{b^2}+\frac{c^2}{a^2b^2}+\frac{b^2+c^2-a^2}{a^2bc}\right)=-4(b^2+c^2-a^2)\left(\frac{1}{a^2b^2}+\frac{1}{a^2bc} \right) & \leq 0.
\end{align*}

By applying Theorem \ref{maximumprinciple1}, there exist some constants $C_1,C_2 > 0$ which depend only on the data on $\mathscr{P}S^1_T$ such that  $|\frac{\zeta}{c}|\leq C_1$ and $|\frac{\zeta}{b}|\leq C_2$. If we choose $\tilde{C}_0=\max(C_1,C_2)$, then $|b-c| \leq \tilde{C} \min(b,c)$ for all $(z,t) \in S^1_T$. Similarly, by applying Theorem \ref{maximumprinciple1} to equations (\ref{ecc5}) and (\ref{ecc6}) as well as noting that $c \geq b$, there exists a constant $\hat{C}_0 >0$ such that $|a-c| \leq \hat{C}_0 \min(a,c)$ for all $(z,t) \in S^1_T$.
\end{proof}

In fact, these equations tell us more:

\begin{mylem} Suppose that at time $t=0$ we have $0 < a$ and $ 1 \leq \displaystyle \max_s \textstyle \left(\frac{c}{a} \right) \leq \lambda $ for some constant $\lambda \geq 1$. Then $ 1 \leq \frac{c}{a} \leq\lambda$ for all $(z,t) \in S^1_T$.  \label{moremax1} \hfill $\blacksquare$
\end{mylem}

\begin{proof}
Since $\frac{\chi}{a}=\frac{a-c}{a}=1-\frac{c}{a}$, the evolution equation for $\frac{c}{a}$ is obtained from equation (\ref{ecc5}) as:
\begin{align*}\partial_t \left(\frac{c}{a} \right)=-\partial_t \left(\frac{\chi}{a} \right)&=-\left(\frac{\chi}{a} \right)''-\left(2\frac{a'}{a}+\frac{b'}{b}\right)\left(\frac{\chi}{a} \right)'+4(a^2+c^2-b^2)\left(\frac{1}{a^2b^2}+\frac{1}{ab^2c}\right)\left(\frac{\chi}{a} \right) 
\\ &= \left(\frac{c}{a} \right)''+\left(2\frac{a'}{a}+\frac{c'}{c}\right)\left(\frac{c}{a} \right)'+4(a^2+c^2-b^2)\left(\frac{1}{a^2b^2}+\frac{1}{ab^2c}\right)\left(1-\frac{c}{a} \right)
\\ &\leq \left(\frac{c}{a} \right)''+\left(2\frac{a'}{a}+\frac{c'}{c}\right)\left(\frac{c}{a} \right)',
\end{align*}
since $1-\frac{c}{a} \leq 0$. By applying Theorem \ref{maximumprinciple1}, we conclude that $ \frac{c}{a} \leq \displaystyle{\max_{s,t=0}} \textstyle \left(\frac{c}{a} \right) =\lambda$.
\end{proof}

\section{Proof of Theorem 1.2}
In this section, we now aim to prove Theorem 1.2. The strategy to prove Theorem 1.2 is to prove several lemmas which bound the metric quantities and trace the spatial minimum of the smallest metric quantity as time progresses. We then show that this spatial minima of the smallest metric quantity can be bounded from above and below by linear functions. Finally, by computing the norm of the Riemann curvature tensor explicitly and using the bound obtained above, we conclude that the singularity that develops during the flow is of Type I. 

Since we have the ordering $a(z,t) \leq b(z,t) \leq c(z,t)$ at all points $(z,t) \in S^1_T$, we wish to study the behaviour of the minima of the function $a$. Intuitively, one can think of the quantity $a$ as the smallest ``radius'' of the $\text{SU}(2)$ fibre at $z$. We define the quantity $\check{a}$ that will be used in our study for the neckpinching of the manifold by:
\begin{align*}\check{a}(t)&=\min_s(a(s,t)).
\end{align*}

Necessarily, we have $\check{a}(t)  >0$ as long as the Ricci flow exists and at this point, $a'(s,t)=0$ and $a''(s,t) \geq 0$. By picturing a neckpinch, we expect the quantity $a$ would vanish at the singular time; that is, $\check{a}(T)=0$. Furthermore, the spatial global minima, denoted by $\check{a}(t)$, is a Lipschitz continuous function of time \cite{HR4}. 

\begin{mylem} If there exists a $T < \infty$ such that $\check{a}(T)=0$, then $\check{a}(t)^2 \leq 4(T-t)$. \label{tracecritical} \hfill $\blacksquare$
\end{mylem}

\begin{proof}  We use the original coordinate $z \in [0,2\pi)$. Since the global spatial minima of the quantity $a$ is a Lipschitz continuous function of time, we may define the derivative $\frac{d }{dt}\check{a}(t)$ in the sense of forward difference quotients as described in \cite{HR4}. Defining $z(t)$ to be a position of the global spatial minima of $a$ at time $t$, we have:
\begin{align}
\frac{d }{dt}\check{a}(t)&=\partial_t a(z(t),t)=  a''+a'\left(\frac{b'}{b}+\frac{c'}{c} \right) - 2a\left(\frac{a^4-(b^2-c^2)^2}{(abc)^2}\right) \nonumber
\\ &\geq -2a\left(\frac{a^4-(b^2-c^2)^2}{(abc)^2}\right) \geq -\frac{2a^3}{b^2c^2}=-\frac{2}{a}\frac{a^2}{b^2}\frac{a^2}{c^2}. \label{acheck2}
\end{align}

Next, by the ordering $a \leq b \leq c$ we have $\frac{a^2}{b^2}\frac{a^2}{c^2} \leq 1$. Along with the inequality (\ref{acheck2}), we deduce:
\begin{align*}
\frac{d}{dt}\check{a} \geq -\frac{2}{\check{a}}\quad \Rightarrow \quad \frac{d}{dt} (\check{a}^2) \geq -4.
\end{align*}

 By integrating the differential inequality from $t$ to $T$ and using the fact that $\check{a}(T)=0$, we have  $\check{a}(t)^2 \leq 4(T-t)$ for almost every $t \in [0,T)$.
\end{proof} 

The lemma above gives us a corollary:

\begin{mycor} The singular time $T$ for the Ricci flow is bounded below by $T \geq \frac{\check{a}(0)^2}{4}$. \hfill $\blacksquare$
\end{mycor}

Apart from the upper bound for $\check{a}(t)$, we can also find a similar bound for $\displaystyle \hat{c}(t)=\max_s(c(s,t))$. Similar as before, this quantity is a Lipschitz continuous function of time \cite{HR4}. From equation (\ref{riccipde3}), along the path for the maximum of $c$ we have:
\begin{equation}\frac{d}{dt} \hat{c} \leq  - 2c\left(\frac{c^4-(a^2-b^2)^2}{(abc)^2}\right)=-\frac{2}{c}\left(\frac{(c^4-b^4)+a^2(2b^2-a^2)}{(ab)^2}\right)\leq 0, \label{bdcmax} \end{equation}
which implies that the maximum of $c$ over $S^1$ decreases over time. In fact, we have the following estimate:

\begin{mylem} The maximum of $c$ over $S^1$ is bounded above by $\hat{c}(t)^2 \leq -4t +\hat{c}(0)^2$. \label{cmax} \hfill $\blacksquare$
\end{mylem}

\begin{proof}
From inequality (\ref{bdcmax}), we study the function $P(x,y)=2+\frac{1}{x^2y^2}-\frac{x^2}{y^2}-\frac{y^2}{x^2}$ since the left hand side of (\ref{bdcmax}) is $-\frac{2}{c}P\left(\frac{a}{c},\frac{b}{c}\right)$. Let $c \leq \lambda a$ at the initial time. Thus, by Lemmas \ref{order} and \ref{moremax1}, we have the ordering $a \leq b \leq c \leq \lambda a$ at all time. Thus, it is sufficient to restrict the domain of the function $P(x,y)$ to $x \in [\lambda^{-1},y]$ and $y \in [x,1]$. We call this domain $\Pi_\lambda \subset \mathbb{R}^2$.  

\begin{figure}[H]\centering \captionsetup{justification=centering}
\begin{tikzpicture}[thick]
      \draw[->] (-0.5,0) -- (3.5,0) node[right] {$x$};
      \draw[->] (0,-0.5) -- (0,3.8) node[above] {$y$};
      \draw[white,fill=black!10] (0.5,0.5) -- (3,3) -- (0.5,3) -- cycle;
      \draw[->-] (3,3) -- (0.5,0.5);
      \draw[->-] (0.5,3) -- (0.5,0.5);
      \draw[->-] (3,3) -- (0.5,3);
      \draw[-] (-0.1,3) -- (0.1,3);
      \node[left] at (0,3) {$1$};
      \draw[-] (-0.1,0.5) -- (0.1,0.5);
      \node[left] at (0,0.5) {$\frac{1}{\lambda}$};
      \draw[-] (0.5,-0.1) -- (0.5,0.1);
      \node[below] at (0.5,0) {$\frac{1}{\lambda}$};
      \node[below, rotate=45] at (1.85,1.6) {$y=x$};
      \node[] at (1.3,2.2) {$\Pi_\lambda$};
    \end{tikzpicture}  
\caption[Domain $\Pi_\lambda$ of the function $P(x,y)$.]{Domain $\Pi_\lambda$ of the function $P(x,y)$. \\ The arrows indicate the direction $P(x,y)$ increases.} \label{pilambda}
\end{figure}
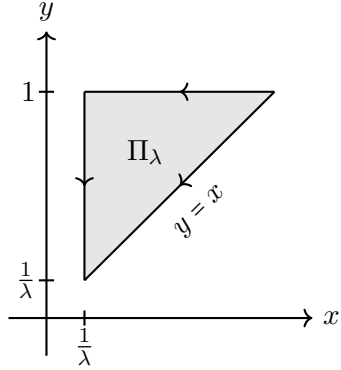

Our aim now is to minimise the function $P$ over the domain $\Pi_\lambda$. Since $\text{grad}(P)=\frac{2}{x^2y^2}(\frac{y^4-x^4-1}{x},\frac{x^2-y^2-1}{y})$, we conclude that the there are no global critical points of $P$ over $\mathbb{R}^2$. Necessarily, the minimum of $P(x,y)$ on $\Pi_\lambda$ lies on $\partial \Pi_\lambda$. Along the boundary of $\Pi_\lambda$ we have:
\begin{align*}
P(x,x)&=\frac{1}{x^4},
\\ P(x,1) &= 2-x^2,
\\ P(\lambda^{-1},y) &= \frac{\lambda^2}{y^2}-\frac{1}{\lambda^2y^2}-\lambda^2y^2+2.
\end{align*}

We note that the functions $P(x,x)$ and $P(x,1)$ are strictly decreasing as $x$ increases from $\lambda^{-1}$ to $1$. Furthermore, since $\lambda \geq 1$, $\partial_yP(\lambda^{-1},y)=\frac{2}{\lambda^2y^3}(1-\lambda^4y^4-\lambda^4)<0$ for any $y \in [\lambda^{-1},1]$. So $P(\lambda^{-1},y)$ is a decreasing function of $y$. Thus, the minimum of $P(x,y)$ on $\Pi_\lambda$ is at the point $(1,1)$ which gives us $P(x,y) \geq P(1,1)=1$. By using this in inequality (\ref{bdcmax}), we get:
$$\frac{d}{dt}\hat{c} \leq -\frac{2}{\hat{c}}\quad \Rightarrow \quad \frac{d}{dt} (\hat{c}^2) \leq -4,$$
for all but a finite number of $t \in [0,T)$. By integrating the differential inequality from $0$ to $t$, we have $\hat{c}(t)^2 -\hat{c}(0)^2 \leq -4t$
for almost every $t \in [0,T)$, hence the result.
\end{proof}

This inequality is sharp and is attained when $a=b=c$; that is, it is attained by the shrinking sphere. The lemma above gives us another corollary:

\begin{mycor} The singular time $T$ for the Ricci flow is bounded above by $T \leq \frac{\hat{c}(0)^2}{4}$. \hfill $\blacksquare$
\end{mycor}

A consequence of the lemma above is a refinement of Lemma \ref{moremax1} using Maximum Principle II in Theorem \ref{maximumprinciple2}. 

\begin{mylem} Suppose that at time $t=0$ we have $1 \leq \displaystyle \max_s \textstyle \left(\frac{c}{a} \right) \leq \lambda $ for some constant $\lambda \geq 1$. Then for all $(z,t) \in S^1_T$ we have the following estimate:
\begin{equation}
1 \leq \frac{c^2}{a^2} \leq e^{(\lambda^2-1)}(\lambda^2-1)\left(1-\frac{4t}{\hat{c}(0)^2} \right)^2+1. \tag*{$\blacksquare$}
\end{equation} 
\end{mylem}
\begin{proof}
Consider the evolution equation for $\frac{c}{a}$ from equation (\ref{ecc5}):
\begin{align*}\partial_t \left(\frac{c}{a} \right)= \left(\frac{c}{a} \right)''+\left(2\frac{a'}{a}+\frac{c'}{c}\right)\left(\frac{c}{a} \right)'+4(a^2+c^2-b^2)\left(\frac{1}{a^2b^2}+\frac{1}{ab^2c}\right)\left(1-\frac{c}{a} \right).
\end{align*}

By defining $u:=\frac{c}{a}$, the equation above can be rewritten as:
\begin{align*}\partial_t u= u''+\left(2\frac{a'}{a}+\frac{c'}{c}\right)u'+\frac{4}{a^2}\left(\frac{a^2}{b^2}-1+\frac{c^2}{b^2}+\frac{a^3}{b^2c}+\frac{ac}{b^2}-\frac{a}{c}\right)(1-u).
\end{align*}
Since $1-u <0$, we want to minimise the coefficient of $1-u$ in the equation above. Using the ordering $a \leq b \leq c$ we have:
\begin{align*}\partial_t u&\leq u''+\left(2\frac{a'}{a}+\frac{c'}{c}\right)u'+\frac{4}{a^2}\left(\frac{u+1}{u^3}\right)(1-u)
\\ & \leq u''+\left(2\frac{a'}{a}+\frac{c'}{c}\right)u'+\frac{4}{c^2}\left(\frac{u+1}{u^3}\right)(1-u)
\\ & \leq u''+\left(2\frac{a'}{a}+\frac{c'}{c}\right)u'+\frac{4}{\hat{c}^2}\left(\frac{u+1}{u^3}\right)(1-u)
\\ & \leq u''+\left(2\frac{a'}{a}+\frac{c'}{c}\right)u'+\frac{4}{\hat{c}(0)^2-4t}\left(\frac{u+1}{u^3}\right)(1-u),
\end{align*}
where we used Lemma \ref{cmax} in the last line. We now appeal to Maximum Principle II by solving the following associated ODE:
\begin{align*} \frac{dv}{dt} &= \frac{4}{\hat{c}(0)^2-4t}\left(\frac{1-v^2}{v^3}\right), 
\\ v(0)&=\lambda.  \end{align*}

This is a separable ODE, which we can solve to get:
\begin{align}v^2-\lambda^2+\log\left(\frac{v^2-1}{\lambda^2-1}\right) =2\log\left(1-\frac{4t}{\hat{c}(0)^2} \right) \quad \Rightarrow \quad e^{v^2-\lambda^2}\left(\frac{v^2-1}{\lambda^2-1}\right)=\left(1-\frac{4t}{\hat{c}(0)^2}\right)^2. \label{solode} \end{align}

Maximum Principle II states that $u$ is bounded above by the function $v$ that satisfies equation (\ref{solode}). We can solve equation (\ref{solode}) for $v$ using the Lambert-W function but this is not expressible in terms of elementary functions. For an explicit bound, since $1 \leq u \leq v$, we have:
\begin{equation*}e^{1-\lambda^2}\left(\frac{v^2-1}{\lambda^2-1}\right) \leq e^{v^2-\lambda^2}\left(\frac{v^2-1}{\lambda^2-1}\right)=\left(1-\frac{4t}{\hat{c}(0)^2}\right)^2, \end{equation*}
from which we can conclude the result  $u^2 \leq v^2  \leq e^{(\lambda^2-1)}(\lambda^2-1)\left(1-\frac{4t}{\hat{c}(0)^2} \right)^2+1$.
\end{proof}

%One corollary from this lemma is that if the singular time is large enough, the fibres of $S^1$ eventually round and this reduces to the uniaxial Bianchi IX metric similary studied by Angenent and Knopf \cite{AK}. This can be seen  in the lemma above:
%$$0 \leq u^2-1 \leq e^{\lambda^2-1}(\lambda^2-1) \left(1-\frac{4t}{\hat{c}(0)^2} \right)^2. $$
%
%Supposing that the singular time $T\geq \frac{\hat{c}(0)^2}{4}$, at time $\tau=\frac{\hat{c}(0)^2}{4}$, we have that $u^2-1=0$ or $\frac{c}{a}=1$ for all $z \in S^1$. By the ordering $a \leq b \leq c$, we have $a(z,\tau)=b(z,\tau)=c(z,\tau)$, giving us a uniaxial Bianchi IX metric. Then, we can run the Ricci flow again with $\tau$ as the initial time, where the analysis after this time $\tau$ will be the same as done by Angenent and Knopf.

\begin{figure}[H]
\captionsetup[subfigure]{justification=centering}
\begin{subfigure}{1\textwidth}\centering
   \includegraphics[scale=0.46]{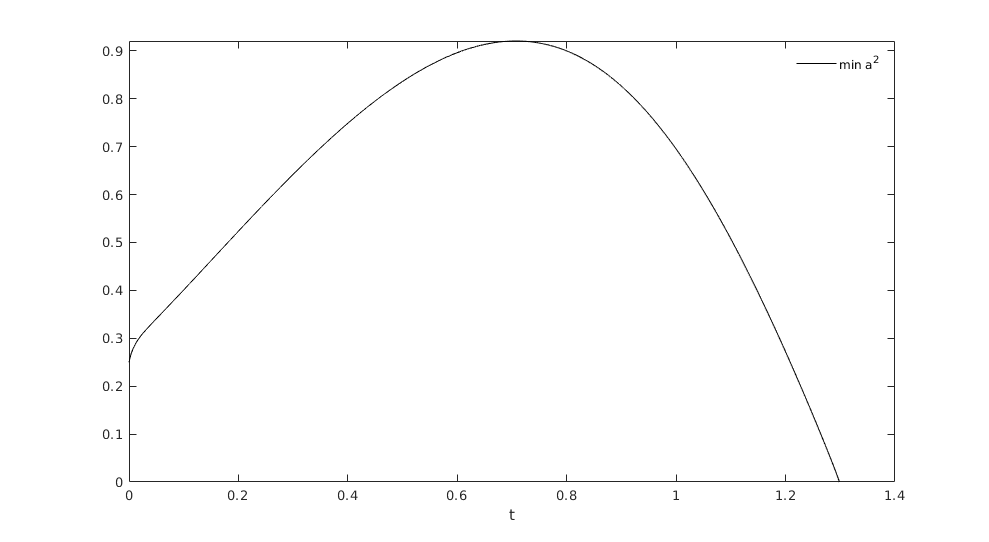}
   \caption{Evolution of $\check{a}^2$ for initial data
    $\phi_0 =1,
a_0 = \cos(z)+1.5,
b_0 = \cos(z)+2.5$, and $
c_0 = \cos(z)+3.5$.}
\end{subfigure}

\begin{subfigure}{1\textwidth}\centering
   \includegraphics[scale=0.46]{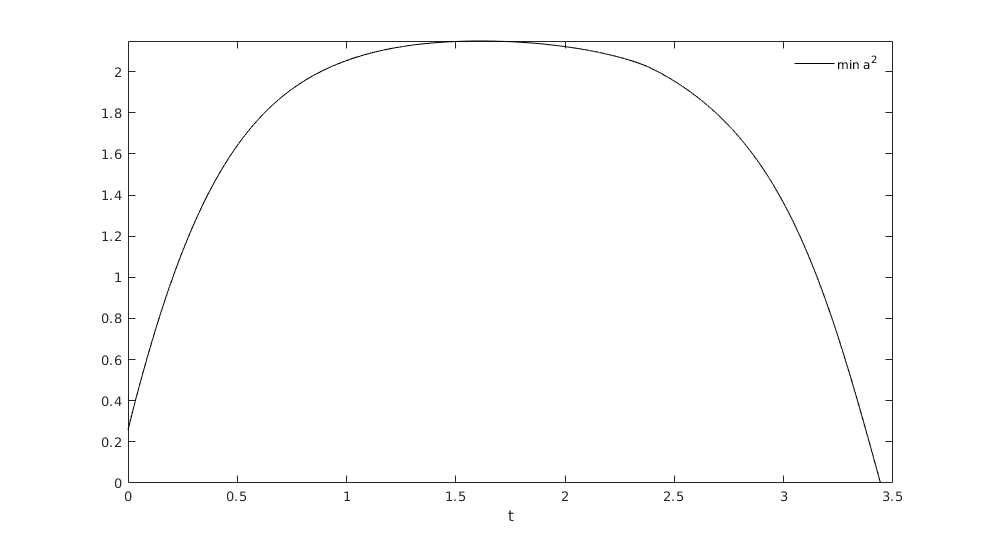}
   \caption{Evolution of $\check{a}^2$ for initial data 
    $\phi_0 =1,
a_0 = \cos(2z)+1.5,
b_0 = \sin(z)+4$, and $
c_0 =6$.}
\end{subfigure}

\begin{subfigure}{1\textwidth}\centering
   \includegraphics[scale=0.46]{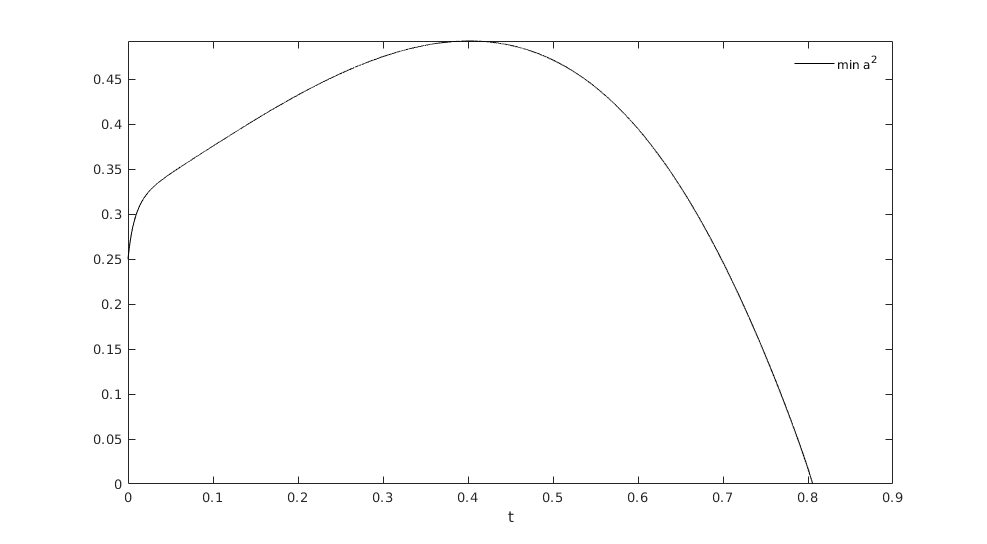}
   \caption{Evolution of $\check{a}^2$ for initial data 
   $\phi_0 =1,
a_0 =\frac{1}{2}\cos(z)+1,
b_0 = \cos(z)+2$, and $
c_0 =2\cos(z)+4$.}
\end{subfigure}
\caption{Numerical plots for the evolution of $\check{a}(t)^2$ for some initial metric.} \label{plot1}
\end{figure}

Now we want to find a lower bound for the quantity $\check{a}(t)$. Numerical simulations in \cite{JMS} provide some evidence that the function $\check{a}(t)^2$ is concave. Some of the plots from \cite{JMS} can be seen in Figure \ref{plot1}. If one can prove that the function $\check{a}(t)^2$ is a concave function of time, then since $\check{a}(0)^2>0$ and $\check{a}(T)^2=0$, we can bound the function $\check{a}(t)^2$ from above and below  by non-negative linear functions with strictly negative gradients such that:
$$D(T-t) \leq \check{a}(t)^2 \leq C(T-t), $$
for some positive constants $C,D>0$.

However, demonstrating the concavity of $\check{a}(t)^2$ by considering its second derivative would be complicated in our setting since the first and second derivatives of the other metric quantities along the minima of $a^2$ would appear in the analysis and we have no control over the sign of these quantities. One way of proceeding with this is to consider the definition of concavity from first principles; that is, by showing that the following inequality holds for any $t_1,t_2 \in [0,T)$ and $p \in [0,1]$:
$$\check{a}(pt_1+(1-p)t_2)^2 \geq p\check{a}(t_1)^2+(1-p)\check{a}(t_2)^2. $$

This is still out of reach of our analysis. However, following the work in \cite{IKS}, by setting an assumption on the initial scalar curvature of the manifold as well as the ratio $\frac{c}{a}$, we can prove the following proposition: 

\begin{myprop} \label{lowerboundofa} Suppose that at time $t=0$ we have $1 \leq \frac{c}{a}< 2$ and the scalar curvature $\text{S}$ at initial time is non-negative; that is, $\displaystyle \min_{s,t=0}(\text{S})\geq 0$. If there exists a finite time $T>0$ such that $\check{a}(T)=0$, then there exists a uniform constant $D >0$ such that $\check{a}(t)^2 \geq D(T-t)$. \hfill $\blacksquare$
\end{myprop}

\begin{proof}
Consider the quantity $\log(abc)$. This quantity evolves according to the PDE:
\begin{align*}
\partial_t \log(abc)&=\frac{\partial_t a}{a}+\frac{\partial_t b}{b}+\frac{\partial_t c}{c}
\\ &=\frac{a''}{a}+\frac{b''}{b}+\frac{c''}{c}+2\left(\frac{a'b'}{ab}+\frac{a'c'}{ac}+\frac{b'c'}{bc} \right)-2(\hat{K}_{12}+\hat{K}_{13}+\hat{K}_{23}) 
\\ &=\frac{a''}{a}+\frac{b''}{b}+\frac{c''}{c}+2\left(\frac{a'b'}{ab}+\frac{a'c'}{ac}+\frac{b'c'}{bc} \right)-2\left(\frac{2a^2b^2+2a^2c^2+2b^2c^2-a^4-b^4-c^4}{a^2b^2c^2} \right).
\end{align*}

Since the scalar curvature $\text{S}$ satisfies $\partial_t \text{S} = \Delta \text{S} +\frac{n}{2}\text{S}^2 \geq \Delta \text{S}$ from equation (\ref{scalarevolution}), we have $\text{S} \geq0$ by minimum principle of parabolic equations. From this and the expression for the scalar curvature $\text{S}$, we deduce:
\begin{align*}
&0 \leq -\left(\frac{a''}{a}+\frac{b''}{b}+\frac{c''}{c}+\frac{a'b'}{ab}+\frac{a'c'}{ac}+\frac{b'c'}{bc}-\frac{2a^2b^2+2a^2c^2+2b^2c^2-a^4-b^4-c^4}{a^2b^2c^2}\right)
\\ \Rightarrow \quad & \frac{a''}{a}+\frac{b''}{b}+\frac{c''}{c} \leq -\left(\frac{a'b'}{ab}+\frac{a'c'}{ac}+\frac{b'c'}{bc}\right)+\frac{2a^2b^2+2a^2c^2+2b^2c^2-a^4-b^4-c^4}{a^2b^2c^2}.
\end{align*}

Thus, we have:
\begin{align*}
\partial_t \log(abc)  \leq  \frac{a'b'}{ab}+\frac{a'c'}{ac}+\frac{b'c'}{bc}+\frac{a^2}{b^2c^2}+\frac{b^2}{a^2c^2}+\frac{c^2}{a^2b^2}-\frac{2}{a^2}-\frac{2}{b^2}-\frac{2}{c^2}.
\end{align*}

Furthermore, we have:
\begin{align*}
\frac{a'b'}{ab}+\frac{a'c'}{ac}+\frac{b'c'}{bc}&=\frac{1}{2}(\log(abc))'\left(\frac{a'}{a}+\frac{b'}{b}+\frac{c'}{c}\right)-\frac{1}{2}\left(\left(\frac{a'}{a}\right)^2+\left(\frac{b'}{b}\right)^2+\left(\frac{c'}{c}\right)^2 \right)
\\ & \leq \frac{1}{2}(\log(abc))'\left(\frac{a'}{a}+\frac{b'}{b}+\frac{c'}{c}\right).
\end{align*}

So, we acquire the inequality:
\begin{align*}
\partial_t \log(abc) &\leq  \frac{1}{2}(\log(abc))'\left(\frac{a'}{a}+\frac{b'}{b}+\frac{c'}{c}\right) +\frac{a^2}{b^2c^2}+\frac{b^2}{a^2c^2}+\frac{c^2}{a^2b^2}-\frac{2}{a^2}-\frac{2}{b^2}-\frac{2}{c^2}
\\ \Leftrightarrow \qquad \quad \partial_t (abc) &\leq \left(\frac{(abc)'}{2abc}\left(\frac{a'}{a}+\frac{b'}{b}+\frac{c'}{c}\right) +\frac{a^2}{b^2c^2}+\frac{b^2}{a^2c^2}+\frac{c^2}{a^2b^2}-\frac{2}{a^2}-\frac{2}{b^2}-\frac{2}{c^2}\right) abc.
\end{align*}

Since $S^1$ is compact, necessarily the global minima of $abc$ over $s$ satisfies $(abc)'=0$. By using the same argument as in Lemma \ref{tracecritical}, along the path $(z(t),t)$ for the global minima of $abc$ over $S^1$, we have:
\begin{align}\frac{d}{dt}abc &\leq  \left(\cancel{\frac{(abc)'}{2abc}\left(\frac{a'}{a}+\frac{b'}{b}+\frac{c'}{c}\right)}+\frac{a^2}{b^2c^2}+\frac{b^2}{a^2c^2}+\frac{c^2}{a^2b^2} -\frac{2}{a^2}-\frac{2}{b^2}-\frac{2}{c^2}\right)abc, \label{abc}
 \end{align}
 for all but a finite number of $t \in [0,T)$.

Following the proof by Isenberg, Knopf, and \v Se\v sum \cite{IKS}, we wish to bound the term $\frac{a^2}{b^2c^2}+\frac{b^2}{a^2c^2}+\frac{c^2}{a^2b^2} -\frac{2}{a^2}-\frac{2}{b^2}-\frac{2}{c^2} $ by a negative multiple of $\frac{1}{c^2}$; that is, for some $\mathfrak{C}>0$, we want:
\begin{align}
\frac{a^2}{b^2c^2} \left( 1+\left(\frac{b^2}{a^2}-\frac{c^2}{a^2} \right)^2-2\frac{b^2}{a^2}-2\frac{c^2}{a^2}\right) \leq -\frac{\mathfrak{C}}{c^2}. \label{boundabc}
\end{align}

This motivates us to study the polynomial $P(x,y)=1+(x^2-y^2)^2-2x^2-2y^2=x^4+y^4-2x^2y^2-2x^2-2y^2+1$ since the left hand side of (\ref{boundabc}) is  a positive multiple of $P\left(\frac{b}{a},\frac{c}{a}\right)$. As a result of the ordering $a \leq b \leq c \leq \lambda a$, we restrict the domain of the polynomial $P(x,y)$ to $x \in [1,y]$ and $y \in [x,\lambda]$. We call this domain $\Omega_\lambda \subset \mathbb{R}^2$.  

\begin{figure}[H]\centering  \captionsetup{justification=centering}
\begin{tikzpicture}[thick]
      \draw[->] (-0.5,0) -- (3.5,0) node[right] {$x$};
      \draw[->] (0,-0.5) -- (0,3.8) node[above] {$y$};
      \draw[white,fill=black!10] (0.5,0.5) -- (3,3) -- (0.5,3) -- cycle;
      \draw[->-] (3,3) -- (0.5,0.5);
      \draw[-] (0.5,0.5) -- (0.5,3);
      \draw[->-] (3,3) -- (0.5,3);
      \draw[-] (-0.1,3) -- (0.1,3);
      \node[left] at (0,3) {$\lambda$};
      \draw[-] (-0.1,0.5) -- (0.1,0.5);
      \node[left] at (0,0.5) {$1$};
      \draw[-] (0.5,-0.1) -- (0.5,0.1);
      \node[below] at (0.5,0) {$1$};
      \node[below, rotate=45] at (1.85,1.6) {$y=x$};
      \node[] at (1.3,2.2) {$\Omega_\lambda$};
    \end{tikzpicture}  
\caption[Domain $\Omega_\lambda$ of the function $P(x,y)$.]{Domain $\Omega_\lambda$ of the function $P(x,y)$. \\ The arrows indicate the direction where $P(x,y)$ increases.} \label{omegalambda}
\end{figure}
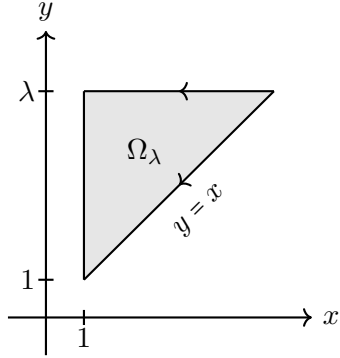

By examining $\text{grad}(P)=4(x(x^2-y^2-1),-y(x^2-y^2+1))$, we conclude that the critical points of $P$ are $(0,0),(0,\pm 1)$, and $(\pm 1,0)$, all of which lie outside of $\Omega_\lambda$. Thus, the maximum of $P(x,y)$ on $\Omega_\lambda$ lies on $\partial \Omega_\lambda$. Along the boundary of $\Omega_\lambda$ we have:
\begin{align*}
P(x,x)&=1-4x^2,
\\ P(1,y) &= y^4-4y^2,
\\ P(x,\lambda) &= 1+x^4+  \lambda^4-2\lambda^2x^2-2x^2-2\lambda^2.
\end{align*}

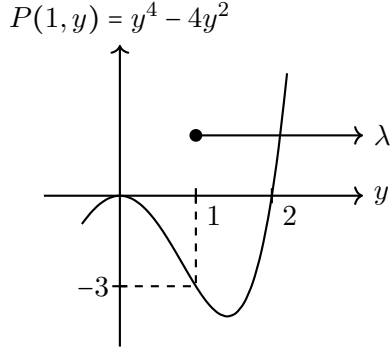
\begin{figure}[H]\centering
\begin{tikzpicture}[thick]
      \draw[->] (-1,0) -- (3.2,0) node[right] {$y$};
      \draw[->] (0,-2) -- (0,2) node[above] {$P(1,y)=y^4-4y^2$};
      \filldraw (1,0.8) circle (2pt);
      \draw[->] (1,0.8) -- (3.2,0.8) node[right] {$\lambda$};
      \draw[scale=1,domain=-0.5:2.2,smooth,variable=\x,black] plot ({\x},{(\x*\x*\x*\x-4*\x*\x)/2.5});
      \draw[-] (1,-0.1) -- (1,0.1);
      \draw[-] (2,-0.1) -- (2,0.1);
      \draw[-] (-0.1,-3/2.5) -- (0.1,-3/2.5);
      %\draw[-] (-0.1,-4/2.5) -- (0.1,-4/2.5);
      \node[left] at (0,-3/2.5) {$-3$  };
      %\node[left] at (0,-4/2.5) {$-4$  };
      \node[below right] at (1,0) {$1$  };
      \node[below right] at (2,-0) {$2$  };
      \draw[dashed] (1,0) -- (1,-3/2.5) -- (0,-3/2.5);
    \end{tikzpicture}
    \caption{Plot of $P(1,y)=y^4-4y^2$.}
    \label{plot}
\end{figure}

For any $\lambda >1$, the polynomials $P(x,x)$ and $P(x,\lambda)$ are strictly decreasing as $x$ increases from $1$ to $\lambda$. Thus, necessarily, the maximum of $P(x,y)$ on $\Omega_\lambda$ lies on the line segment connecting $(1,1)$ to $(1,\lambda)$. From the plot in Figure \ref{plot}, we note that the maximum of $P(1,y)$ for $ y \in [1,\lambda]$ is negative if $\lambda <2$. Thus, a sufficient and necessary condition for the maximum of $P(x,y)$ to be negative in $\Omega_\lambda$ is $\lambda <2$. Furthermore, $\displaystyle \max_{\Omega_\lambda}(P(x,y))= \max(-3,\lambda^4-4\lambda^2):=-\lambda_0$ for some $\lambda_0>0$.

Hence, if we have $\frac{c}{a}<2$ at time $t=0$, by compactness of $S^1$, there exists a $1<\lambda <2$ such that $\frac{c}{a} \leq \lambda$ at time $t=0$. By virtue of  Lemma \ref{moremax1}, we have $\frac{c}{a} \leq \lambda$ for all $S^1_T$ and thus:
\begin{align*}
 1+\left(\frac{b^2}{a^2}-\frac{c^2}{a^2} \right)^2-2\frac{b^2}{a^2}-2\frac{c^2}{a^2}=P\left( \frac{b}{a},\frac{c}{a}\right) &\leq -\lambda_0.
\\ \Rightarrow \quad \quad \frac{a^2}{b^2c^2} \left( 1+\left(\frac{b^2}{a^2}-\frac{c^2}{a^2} \right)^2-2\frac{b^2}{a^2}-2\frac{c^2}{a^2}\right) &\leq -\lambda_0\frac{a^2}{b^2c^2} < -\frac{\lambda_0}{\lambda^2c^2},
\end{align*}
since the ordering $a \leq b \leq c \leq \lambda a <2a$ is preserved under the Ricci flow. From inequality (\ref{abc}), along the minimum of $abc$ we have:
\begin{align*}\frac{d}{dt}abc &\leq  \left(-\frac{\lambda_0}{\lambda^2 c^2}\right)abc=-\frac{\lambda_0}{\lambda^2}\frac{ab}{c} \leq -\frac{\lambda_0}{\lambda^\frac{10}{3}}(abc)^\frac{1}{3},
 \end{align*}
where we used the fact that $a \leq b \leq c \leq \lambda a$ and Lemma \ref{moremax1} for the last inequality.
%Note that $\displaystyle \hat{c}(t)=\max_s(c(s,t))$ is a non-increasing function of time since
%$$\frac{d}{dt} \hat{c} \leq  - 4c\left(\frac{c^4-(a^2-b^2)^2}{2(abc)^2}\right)=-2c\left(\frac{c^4-b^4+a^2(2b^2-a^2)}{(abc)^2} \right) \leq 0. $$
%
%Furthermore, by our assumption at $t=0$ and compactness of the interval $I=[-1,1]$, for some $0<D_0 < \frac{1}{4}$ small enough, we have $\displaystyle \min_{s,t=0}(\text{S})\cdot \max_{s,t=0}(c^2) \geq -\frac{1}{4} + D_0$, and therefore:
%\begin{align*}
%2s_0 \geq \min_{s,t=0}(\text{S}) \geq \frac{1}{\displaystyle \max_{s,t=0}(c^2)}\left(-\frac{1}{4}+D_0\right) \geq \frac{1}{c^2}\left(-\frac{1}{4}+D_0\right) \quad \Rightarrow \quad -\frac{1}{8c^2}-s_0 \leq -\frac{D_0}{2c^2}.
%\end{align*}

%Substituting this in the inequality (\ref{ac2}), along the local minimum values of $a$, we have:
%$$\frac{d}{dt}\overline{abc}(z(t),t) \leq \left(-\frac{1}{8c^2}-s_0 \right)ac^2 \leq -\frac{D_0a}{2} $$

By setting $m(t)=abc(z(t),t)$,  for all but a finite number of $t \in [0,T)$ we have:
\begin{align*}
\frac{d}{dt}m(z(t),t) \leq -D_1m^\frac{1}{3},
\end{align*}
where $D_1 =\frac{\lambda_0}{\lambda^\frac{10}{3}}>0$ is some constant. Integrating this inequality from an arbitrary $t<T$ to $T$ and noting that $m(T)=0$ since $\check{a}(T)=0$, we have $m(t)^\frac{2}{3} \geq \frac{2D_1}{3}(T-t)$. Furthermore, by using Lemma \ref{moremax1} we obtain the inequality:
$$\check{a}(t)^3 \geq \frac{1}{\lambda^2}abc(z(t),t) = \frac{1}{\lambda^2}m(t) \geq \frac{2\sqrt{6}D_1^\frac{3}{2}}{9\lambda^2}(T-t)^\frac{3}{2}, $$
and thus $\check{a}(t)^2 \geq D(T-t)$ for the constant $D=\frac{2}{3}\frac{\lambda_0}{\lambda^\frac{14}{3}} >0$ where $\lambda_0=\min(3,4\lambda^2-\lambda^4)$.
\end{proof}

The next step is to find the bounds for  the first and second derivatives of the metric components. This is required since they appear in the curvature quantities of the metric $g$. We prove:

\begin{mylem} Suppose that $a \leq b \leq c \leq \lambda a$ for some $1 < \lambda <2$ at time $t=0$. Then for all $(s,t) \in S^1_T$ there exist constants $C_1(\lambda),C_2(\lambda), C_3(\lambda)>0$ such that: \label{boundaprimebprimecprime}
\begin{align*}
|a'(s,t)| &\leq C_1(\lambda) \leq \max\left(\frac{280\sqrt{3}}{9},\max_s(|a'(s,0)|) \right),
\\ |b'(s,t)| &\leq C_2(\lambda) \leq \max\left(\frac{4\sqrt{57}}{3},\max_s(|b'(s,0)|)\right),
\\ |c'(s,t)| &\leq C_3(\lambda) \leq \max\left(\frac{10\sqrt{93}}{9},\max_s(|c'(s,0)|)\right).  \tag*{$\blacksquare$}
\end{align*}
\end{mylem}
\begin{proof} We recall the evolution equation for $a'$ in (\ref{aprime}):
\begin{align}
\partial_t(a')&= a''' +a''\left(\frac{b'}{b}+\frac{c'}{c}-\frac{a'}{a} \right)-a'\left(\frac{(b')^2}{b^2}+\frac{(c')^2}{c^2}+\frac{6a^4+2(b^2-c^2)^2}{(abc)^2} \right) \nonumber
\\ & \qquad +4 \left(\frac{a^3b'}{b^3c^2}+\frac{a^3c'}{b^2c^3}-\frac{c^2b'}{ab^3}-\frac{b^2c'}{ac^3}+\frac{cc'}{ab^2}+\frac{bb'}{ac^2}\right).\label{derivaderiv} \end{align}

We denote $\displaystyle \tilde{a}'=\max_s(a'(s,t))$. By using the same method as in Lemma \ref{tracecritical}, we have $\tilde{a}'' \leq 0$ and $\tilde{a}'=0$ along the path $(z(t),t)$ for the maximum value of $a'$. We substitute this in (\ref{derivaderiv}):
\begin{align}\frac{d}{dt}\tilde{a}' &\leq -\tilde{a}'\left(\frac{(b')^2}{b^2}+\frac{(c')^2}{c^2}+\frac{6a^4+2(b^2-c^2)^2}{(abc)^2} \right)+4 \left(\frac{a^3b'}{b^3c^2}+\frac{a^3c'}{b^2c^3}-\frac{c^2b'}{ab^3}-\frac{b^2c'}{ac^3}+\frac{cc'}{ab^2}+\frac{bb'}{ac^2}\right) \nonumber
\\ &= -\tilde{a}'\left(\frac{(b')^2}{b^2}+\frac{(c')^2}{c^2}+\frac{6a^2}{b^2c^2}+\frac{2b^2}{a^2c^2}+\frac{2c^2}{a^2b^2} \right) +\frac{4\tilde{a}'}{a^2} \nonumber
\\ & \qquad +4 \left(\frac{a^3b'}{b^3c^2}+\frac{a^3c'}{b^2c^3}-\frac{c^2b'}{ab^3}-\frac{b^2c'}{ac^3}+\frac{cc'}{ab^2}+\frac{bb'}{ac^2}\right). \label{ineqineq}
\end{align}

To deal with the terms in the second bracket on the right hand side of the inequality (\ref{ineqineq}) we use Young's inequality on each of them. So, for any $\varepsilon_i>0$ for $i=1,2,3,4,5,6$, we have:
\begin{align*}
\frac{a^3b'}{b^3c^2} &\leq \frac{\varepsilon_1}{2}\frac{(b')^2}{b^2}+\frac{1}{2\varepsilon_1}\frac{a^6}{b^4c^4} \leq \frac{\varepsilon_1}{2}\frac{(b')^2}{b^2}+\frac{1}{2\varepsilon_1}\frac{a^2}{b^2c^2},
\\ \frac{a^3c'}{b^2c^3} &\leq \frac{\varepsilon_2}{2}\frac{(c')^2}{c^2}+\frac{1}{2\varepsilon_2}\frac{a^6}{b^4c^4}\leq \frac{\varepsilon_2}{2}\frac{(c')^2}{c^2}+\frac{1}{2\varepsilon_2}\frac{a^2}{b^2c^2},
\\ -\frac{c^2b'}{ab^3} &\leq \frac{\varepsilon_3}{2}\frac{(b')^2}{b^2}+\frac{1}{2\varepsilon_3}\frac{c^4}{a^2b^4}\leq \frac{\varepsilon_3}{2}\frac{(b')^2}{b^2}+\frac{1}{2\varepsilon_3}\frac{\lambda^2c^2}{a^2b^2},
\\ -\frac{b^2c'}{ac^3} &\leq \frac{\varepsilon_4}{2}\frac{(c')^2}{c^2}+\frac{1}{2\varepsilon_4}\frac{b^4}{a^2c^4}\leq \frac{\varepsilon_4}{2}\frac{(c')^2}{c^2}+\frac{1}{2\varepsilon_4}\frac{b^2}{a^2c^2},
\\ \frac{bb'}{ac^2} &\leq \frac{\varepsilon_5}{2}\frac{(b')^2}{b^2}+\frac{1}{2\varepsilon_5}\frac{b^4}{a^2c^4}\leq \frac{\varepsilon_5}{2}\frac{(b')^2}{b^2}+\frac{1}{2\varepsilon_5}\frac{b^2}{a^2c^2},
\\ \frac{cc'}{ab^2} &\leq \frac{\varepsilon_6}{2}\frac{(c')^2}{c^2}+\frac{1}{2\varepsilon_6}\frac{c^4}{a^2b^4}\leq \frac{\varepsilon_6}{2}\frac{(c')^2}{c^2}+\frac{1}{2\varepsilon_6}\frac{\lambda^2c^2}{a^2b^2},
\end{align*}
where we used the fact that the ordering $a \leq b \leq c \leq \lambda a$ is preserved for all $t >0$. For a crude bound, we set $\varepsilon_1=\varepsilon_2$, $\varepsilon_3=\varepsilon_6$, and $\varepsilon_4=\varepsilon_5$. Therefore, we deduce:
%\begin{align*}
%4 \left(\frac{a^3b'}{b^3c^2}+\frac{a^3c'}{b^2c^3}-\frac{c^2b'}{ab^3}-\frac{b^2c'}{ac^3}+\frac{cc'}{ab^2}+\frac{bb'}{ac^2}\right) &\leq 2\left(\frac{1}{\varepsilon_1}+\frac{1}{\varepsilon_3}+\frac{1}{\varepsilon_5} \right)\frac{(b')^2}{b^2}+2\left(\frac{1}{\varepsilon_2}+\frac{1}{\varepsilon_4}+\frac{1}{\varepsilon_6} \right)\frac{(c')^2}{c^2}
%\\ & \quad +\left(\frac{1}{\varepsilon_1}+\frac{1}{\varepsilon_2} \right)\frac{2a^2}{b^2c^2} +\left(\frac{1}{\varepsilon_4}+\frac{1}{\varepsilon_5} \right)\frac{2b^2}{a^2c^2} +\left(\frac{1}{\varepsilon_3}+\frac{1}{\varepsilon_6} \right)\frac{2\lambda^2c^2}{a^2b^2} 
%\end{align*}
\begin{align}
4 \left(\frac{a^3b'}{b^3c^2}+\frac{a^3c'}{b^2c^3}-\frac{c^2b'}{ab^3}-\frac{b^2c'}{ac^3}+\frac{cc'}{ab^2}+\frac{bb'}{ac^2}\right) &\leq 2\left(\varepsilon_1+\varepsilon_3+\varepsilon_5 \right)\left(\frac{(b')^2}{b^2}+\frac{(c')^2}{c^2}\right)+\frac{1}{\varepsilon_1}\frac{4a^2}{b^2c^2} \nonumber
\\ & \qquad +\frac{1}{\varepsilon_5}\frac{4b^2}{a^2c^2} +\frac{1}{\varepsilon_3}\frac{4\lambda^2c^2}{a^2b^2} . \label{e1e3e5}
\end{align}

We want the right hand side of inequality (\ref{e1e3e5}) to be a positive multiple of the first bracket in the inequality (\ref{ineqineq}). In order to choose the constants $\varepsilon_1,\varepsilon_3,\varepsilon_5 >0$, we solve the following simultaneous equations for some constant $\mathfrak{C}>0$:
\begin{align*}2(\varepsilon_1+\varepsilon_3+\varepsilon_5)=\mathfrak{C}, \quad 4=6\mathfrak{C}\varepsilon_1, \quad 4=2\mathfrak{C}\varepsilon_5, \quad \text{and} \quad 4\lambda^2=2\mathfrak{C}\varepsilon_3.
\end{align*}

The system above has a solution $\mathfrak{C}=\sqrt{\frac{16}{3}+4\lambda^2}$ for some $\varepsilon_1,\varepsilon_3,\varepsilon_5>0$. Thus, we obtain:
\begin{align*}\frac{d}{dt}\tilde{a}' &\leq (\mathfrak{C}-\tilde{a}')\left(\frac{(b')^2}{b^2}+\frac{(c')^2}{c^2}+\frac{6a^2}{b^2c^2}+\frac{2b^2}{a^2c^2}+\frac{2c^2}{a^2b^2} \right)+\frac{4\tilde{a}'}{a^2} .
\end{align*}

Whenever $\tilde{a}' \geq \mathfrak{C}$, the first term would be negative whereas the second term will be positive. We want the right hand side to be negative, so we now aim to get rid of the positive term  $\frac{4\tilde{a}'}{a^2}$ by adding a large enough negative term to it. By assuming $\tilde{a}' \geq \mathfrak{C}$, we have:
\begin{align*}\frac{d}{dt}\tilde{a}' &\leq  (\mathfrak{C}-\tilde{a}')\left(\frac{(b')^2}{b^2}+\frac{(c')^2}{c^2}\right) +(\mathfrak{C}-\tilde{a}')\left(\frac{6a^2}{b^2c^2}+\frac{2b^2}{a^2c^2}+\frac{2c^2}{a^2b^2} \right)+\frac{4\tilde{a}'}{a^2} 
\\ &= (\mathfrak{C}-\tilde{a}')\left(\frac{(b')^2}{b^2}+\frac{(c')^2}{c^2}\right)+\frac{4}{a^2}\left((\mathfrak{C}-\tilde{a}')\left(\frac{3a^4}{2b^2c^2}+\frac{b^2}{2c^2}+\frac{c^2}{2b^2}\right)+ \tilde{a}'\right).
\end{align*}

We now look at the function of two variables $P(x,y)=\frac{3}{2x^2y^2}+\frac{x^2}{2y^2}+\frac{y^2}{2x^2}$ and determine its minimum value in the region $\Omega_\lambda$ as in Figure \ref{omegalambda} in the proof for Proposition \ref{lowerboundofa}. Upon calculations, we note that the minimum of $P(x,y)$  in $\Omega_\lambda$ occurs at $(x,y)=(\lambda,\lambda)$, which gives us $\displaystyle \min_{\Omega_\lambda}P(x,y) \textstyle =1+\frac{3}{2\lambda^4}$. Hence:
\begin{align*}\frac{d}{dt}\tilde{a}'  &\leq (\mathfrak{C}-\tilde{a}')\left(\frac{(b')^2}{b^2}+\frac{(c')^2}{c^2}\right)+\frac{4}{a^2}\left((\mathfrak{C}-\tilde{a}')\left(1+\frac{3}{2\lambda^4}\right)+ \tilde{a}'\right),
\\ &= (\mathfrak{C}-\tilde{a}')\left(\frac{(b')^2}{b^2}+\frac{(c')^2}{c^2}\right)+\frac{4}{a^2}\left(\left(1+\frac{3}{2\lambda^4}\right)\mathfrak{C}-\frac{3}{2\lambda^4} \tilde{a}'\right).
\end{align*}

Thus, we conclude that $\displaystyle \max_s(a'(s,t))$ cannot exceed $\mathfrak{C}\left(1+\frac{2\lambda^4}{3} \right) \leq \frac{280\sqrt{3}}{9}$ and $\displaystyle \max_s(a'(s,0))$. By running a similar argument on the negative of $a$, we obtain the lower bound $\displaystyle \min_s(a'(s,t)) \textstyle \geq -\max\left(\frac{280\sqrt{3}}{9},\displaystyle \max_s(|a'(s,0)|) \right)$, which gives us the desired bound for $a'$. 

Using arguments identical to the above and noting that the minimum of $P(x,y)$ on the triangles $\{(x,y): \frac{1}{2} \leq x \leq 1, 1 \leq y \leq 2x\}$ and $\{(x,y): \frac{1}{2} \leq x \leq 1, x \leq y \leq 1\}$ are $P(1,\sqrt{2})=2$ and $P(1,1)=\frac{5}{2}$ respectively, similar bounds for $b'$ and $c'$ can be obtained.
\end{proof}

By the work of \v Se\v sum in \cite{SN3}, we note that a finite time singularity occurs at $T < \infty$ only if:
$$\limsup_{t \nearrow T} \max_{p \in M} |\text{Ric}(p,t)|_g=\infty.$$

Furthermore, this singularity is of Type I if we have:
$$\sup_{t \in [0,T)} \max_{p \in M} (T-t)|\text{Rm}(p,t)|_g \leq C < \infty, $$
for some constant $C>0$. Thus, we now want to investigate the behaviour of the norm of the Riemann curvature tensor, which is given by:
$$|\text{Rm}|^2_g=2(K_{01}^2+K_{02}^2+K_{03}^2+K_{12}^2+K_{13}^2+K_{23}^2). $$ 

\begin{mylem} \label{K123} If $1 \leq \frac{c}{a} \leq \lambda$ for some $\lambda <2$ at $t=0$, then for all $p \in M$ we have:
\begin{equation}
|K_{12}(p,t)|+|K_{13}(p,t)|+|K_{23}(p,t)| \leq \frac{C(\lambda)}{a(\pi(p),t)^2},
\end{equation}
where $\pi: S^1 \times S^3 \rightarrow S^1$ is the projection map from the manifold $S^1 \times S^3$ to the $S^1$ base of the foliation. \hfill $\blacksquare$
\end{mylem}

\begin{proof}
The sectional curvatures $K_{12},K_{13}$, and $K_{23}$ are given by:
\begin{align*}
K_{12} &=-\frac{a'b'}{ab}+\frac{(a^2-b^2)^2-3c^4}{(abc)^2}+\frac{2}{a^2}+\frac{2}{b^2},
\\ K_{13} &=-\frac{a'c'}{ac}+\frac{(a^2-c^2)^2-3b^4}{(abc)^2}+\frac{2}{a^2}+\frac{2}{c^2},
\\ K_{23} &=-\frac{b'c'}{bc}+\frac{(b^2-c^2)^2-3a^4}{(abc)^2}+\frac{2}{b^2}+\frac{2}{c^2}.
\end{align*}

Thus, by Lemma \ref{boundaprimebprimecprime} and the fact that $a \leq b \leq c \leq \lambda a$ for all $t \geq 0$, we obtain the bound:
\begin{align*}
|K_{12}|+|K_{13}|+|K_{23}| &\leq \frac{|a'||b'|}{ab}+\frac{|a'||c'|}{ac}+\frac{|b'||c'|}{bc}+\frac{a^2}{b^2c^2}+\frac{b^2}{a^2c^2}+\frac{c^2}{a^2b^2}+\frac{2}{a^2}+\frac{2}{b^2}+\frac{2}{c^2}
\\ & \leq \tilde{C}(\lambda)\left(\frac{1}{ab}+\frac{1}{ac}+\frac{1}{bc}\right)+\frac{1}{c^2}+\frac{1}{a^2}+\frac{1}{\lambda^2 b^2}+\frac{6}{a^2} \leq \frac{C(\lambda)^2}{a^2},
\end{align*}
for some constant $C(\lambda)>0$.
\end{proof}

We now want to bound the sectional curvatures $K_{01},K_{02}$, and $K_{03}$.  These curvatures involve the second order derivatives of the quantities $a,b,$ and $c$. A similar bound as in Lemma \ref{K123} may be obtained, but this requires a bit more work.

\begin{mylem} If $1 \leq \frac{c}{a} \leq \lambda$ for some $\lambda <2$ at $t=0$, then for all $p \in M$ we have:
\begin{equation}
|K_{01}(p,t)|+|K_{02}(p,t)|+|K_{03}(p,t)| \leq \frac{C(\lambda)}{\check{a}(t)^2}. \tag*{$\blacksquare$}
\end{equation}
\end{mylem}

\begin{proof}
 First we recall the evolution equation for $K_{01}$ computed in (\ref{K01eq}):
\begin{align*}
\partial_t K_{01}&=\Delta_g K_{01}+2K_{01}^2-2K_{01}\left( \frac{(b')^2}{b^2}+\frac{(c')^2}{c^2}+\frac{2a^4+2(b^2-c^2)^2}{(abc)^2}\right) \nonumber
\\ &\qquad +2K_{02}\left(\frac{2a^2}{b^2c^2}+\frac{2b^2}{a^2c^2}-\frac{2c^2}{a^2b^2}-\frac{a'b'}{ab} \right)+2K_{03}\left(\frac{2a^2}{b^2c^2}+\frac{2c^2}{a^2b^2}-\frac{2b^2}{a^2c^2}-\frac{a'c'}{ac} \right) \nonumber
\\ & \qquad +2\frac{a'}{a}\bigg(-\frac{(b')^3}{b^3}-\frac{(c')^3}{c^3}+\frac{6aa'}{b^2c^2}+\frac{4bb'}{a^2c^2}+\frac{4cc'}{a^2b^2}-\frac{2a'c^2}{a^3b^2}-\frac{2a'b^2}{a^3c^2}+\frac{4a'}{a^3}  \nonumber
\\ &\qquad -\frac{12a^2b'}{b^3c^2}-\frac{12a^2c'}{b^2c^3}-\frac{4b^2c'}{a^2c^3}-\frac{4c^2b'}{a^2b^3}\bigg)+4a^2\left(\frac{3(b')^2}{b^4c^2}+\frac{4b'c'}{b^3c^3}+\frac{3(c')^2}{b^2c^4} \right) \nonumber
\\ &\qquad -\frac{4}{a}\left(\frac{(c')^2}{ab^2}+\frac{(b')^2}{ac^2}+\frac{3(b')^2c^2}{ab^4}+\frac{3b^2(c')^2}{ac^4}-\frac{4cc'b'}{ab^3}-\frac{4bb'c'}{ac^3}\right).
\end{align*}

We can bound this quantity by using Lemma \ref{boundaprimebprimecprime} and the fact that $a \leq b \leq c < 2a$ for all $t \geq 0$. Thus, for some positive constants $D_1,D_2,D_3,D_4>0$ depending only on $\lambda >0$, we have:
\begin{align*}
\partial_t K_{01} &\leq \Delta_g K_{01}+2K_{01}^2+D_1\frac{K_{01}}{a^2}+D_2\frac{K_{02}}{a^2}+D_3\frac{K_{03}}{a^2}+\frac{D_4}{a^4}
\\ &\leq \Delta_g K_{01}+3K_{01}^2+K_{02}^2+K_{03}^2+\frac{D_0}{a^4},
\end{align*}
where we applied the Cauchy-Schwarz inequality to get $D_0:=D_1^2+D_2^2+D_3^2+D_4$. Furthermore, by using the same estimates and using Young's inequality for the last term in the evolution equation for the quantity $\frac{(a')^2}{a^2}$ obtained from (\ref{aprimeasquared}), we obtain:
\begin{align*}
\partial_t \left(\frac{(a')^2}{a^2} \right)&=\Delta_g \left(\frac{(a')^2}{a^2} \right)-2\frac{(a')^2}{a^2}\left(2\frac{(a')^2}{a^2}+ \frac{(b')^2}{b^2}+\frac{(c')^2}{c^2}+\frac{4(a^4+(b^2-c^2)^2)}{a^2b^2c^2}\right) \nonumber
\\ &\qquad +8\frac{a'}{a^2}\left(\frac{a^3b'}{b^3c^2}+\frac{a^3c'}{b^2c^3}-\frac{c^2b'}{ab^3}-\frac{b^2c'}{ac^3}+\frac{cc'}{ab^2}+\frac{bb'}{ac^2}\right)-2K_{01}^2 -4\frac{(a')^2}{a^2}K_{01},
\\ & \leq \Delta_g \left(\frac{(a')^2}{a^2} \right)+\frac{D_5}{a^4}-K_{01}^2,
\end{align*}
where $D_5>0$ is a positive constant. By repeating this for the evolution equations of the quantities $\frac{(b')^2}{b^2} $ and $\frac{(c')^2}{c^2}$ in (\ref{bprimebsquared}) and (\ref{cprimecsquared}) respectively, for some positive constants $D_6,D_7>0$ we have the inequalities:
\begin{align*}
\partial_t \left(\frac{(b')^2}{b^2} \right)& \leq \Delta_g \left(\frac{(b')^2}{b^2} \right)+\frac{D_6}{a^4}-K_{02}^2,
\\ \partial_t \left(\frac{(c')^2}{c^2} \right)& \leq \Delta_g \left(\frac{(c')^2}{c^2} \right)+\frac{D_7}{a^4}-K_{03}^2.
\end{align*}

Let us define $P_+=K_{01}+3\frac{(a')^2}{a^2}+\frac{(b')^2}{b^2}+\frac{(c')^2}{c^2} \geq K_{01}$. By linearity, the evolution of the quantity $P_+$ is governed by the inequality:
$$\partial_tP_+ \leq \Delta_g P_+ +\frac{D}{a^4}, $$
where $D:=D_0+3D_5+D_6+D_7>0$ is some positive constant. By tracking $\displaystyle \hat{P}_+=\max_s(P_+(s,t))$ using the same argument as in Lemma \ref{tracecritical}, we have:
\begin{equation} \label{pplus} \frac{d}{dt}P_+(z(t),t)) \leq \frac{D}{a^4} \leq \frac{D}{\check{a}^4}, \end{equation}
which implies that $P_+$, and hence $K_{01}$, cannot blow up to $+\infty$ as long as the metric component $a$ is bounded away from $0$.

To bound the sectional curvature $K_{01}$ from below, by using Lemma \ref{boundaprimebprimecprime} and the fact that $a \leq b \leq c < 2a$ for all $t \geq 0$, we note that there exists a constant $C_0>0$ such that:
$$\partial_tK_{01}\geq \Delta_gK_{01}-3K_{01}^2-K_{02}^2-K_{03}^2-\frac{C_0}{a^4}. $$

We now define $P_-=K_{01}-3\frac{(a')^2}{a^2}-\frac{(b')^2}{b^2}-\frac{(c')^2}{c^2} \leq K_{01}$. This quantity satisfies:
$$\partial_tP_- \geq \Delta_g P_- -\frac{C}{a^4}, $$
for some constant $C:=D+3D_5+D_6+D_7>0$. Tracking the minimum of $P_-$, we obtain:
\begin{equation} \label{pminus} \frac{d}{dt}P_-(z(t),t) \geq -\frac{C}{a^4} \geq -\frac{C}{\check{a}^4}. \end{equation}

By a similar argument as before, this implies that $P_-$, and hence $K_{01}$, cannot approach $-\infty$ as long as the metric quantity $a$ is bounded away from $0$. Both of these arguments imply that the sectional curvature $K_{01}$ is bounded on the manifold for times $[0,T)$ where $T$ is the singular time when the quantity $a$ reaches $0$ somewhere.

Since we have shown that $D(T-t)\leq \check{a}(t)^2 \leq C(T-t)$ for some constants $C,D>0$, by substituting the lower bound for $\check{a}$ in (\ref{pplus}) and (\ref{pminus}), we get the inequalities:
\begin{align*}
\frac{d}{dt}\hat{P}_+ &\leq \frac{D}{(T-t)^2},
\\ \frac{d}{dt}\check{P}_- &\geq -\frac{C}{(T-t)^2},
\end{align*}
for some constants $C,D>0$.

By integrating both of these, for some constants $\tilde{C},\tilde{D}>0$, we acquire:
$$|K_{01}| \leq \frac{\tilde{C}}{(T-t)}+\tilde{D} \leq \frac{\tilde{C}+T\tilde{D}}{(T-t)}, $$
for all but a finite number of $t \in [0,T)$.

If we substitute the upper bound for $\check{a}^2$ in the inequality above, we get the required bound for $|K_{01}|$. A similar procedure can be used to bound $|K_{02}|$ and $|K_{03}|$ using equations (\ref{K02eq}) and (\ref{K03eq}). This proves the lemma.
\end{proof}

By putting all of these together, we have:
$$|\text{Rm}|^2_g=2(K_{01}^2+K_{02}^2+K_{03}^2+K_{12}^2+K_{13}^2+K_{23}^2) \leq \frac{3D^2}{\check{a}^4}+ \frac{3C^2}{a^4}\leq \frac{\tilde{C}}{\check{a}^4}, $$ 
for some finite constant $0 < \tilde{C}$. The lower bound for $\check{a}$ in Proposition \ref{lowerboundofa} implies the bound $(T-t)^2|\text{Rm}|^2_g \leq C$ for some finite constant $C>0$. If we take the the supremum over $S^1 \times S^3$ and $t\in [0,T)$ on the left hand side, we can show that the curvature asymptotics satisfy the Type I condition as in (\ref{type1}). 

Finally, we note that the metric constructed by Isenberg, Knopf, and \v Se\v sum \cite{IKS} also satisfies the conditions required for Proposition \ref{lowerboundofa}, so the set of metrics satisfying the sufficient conditions in Theorem \ref{theorem1} is non-empty. Thus, this proves Theorem \ref{theorem1}.

\appendix 
\section{Computation of the Ricci Flow Equations}
We can endow the manifold $N=S^1 \times S^3$ with a metric of the form:
\begin{equation}
g = \phi(z)^2 \, dz^2 +a(z)^2 \omega^1 \otimes \omega^1 + b(z)^2 \omega^2 \otimes \omega^2 + c(z)^2 \omega^3 \otimes \omega^3, \label{ansatzmetric}
\end{equation}
where $\phi,a,b,$ and $c$ are positive $2\pi$-periodic functions on the base manifold $S^1$. Now we wish to calculate the curvature tensors of the metric $g$ in the frame $\{\partial_z=E_0, E_1,E_2,E_3\}$ for $N=S^1 \times S^3$. The first thing we have to consider is the derivatives of the frame vector fields. We write the derivatives of the frame vector fields using the symbols $\Sigma_{\alpha \beta}^\gamma$ for $\alpha, \beta, \gamma \in \{0,1,2,3\}$ by:
$$\nabla_{E_\alpha}E_\beta=\Sigma_{\alpha \beta}^\gamma E_\gamma. $$

Recall that the Christoffel symbols $\Gamma_{ij}^k$ are defined via local coordinates. Thus, the symbols $\Sigma_{\alpha \beta}^\gamma$ are not the same as Christoffel symbols because the vector fields $\{E_0,E_1,E_2,E_3\}$ on $N$ are not coordinate vector fields. We call $\Sigma_{\alpha \beta}^\gamma$ the frame symbols. Using the Koszul formula, we calculate each of the 
frame symbols. We have that:
\begin{align}
2g(E_\gamma,E_\gamma) \Sigma_{\alpha \beta}^\gamma &= 2g(\nabla_{E_\alpha}E_\beta,E_\gamma) \nonumber
\\ &= E_\alpha g(E_\beta,E_\gamma)+E_\beta g(E_\gamma,E_\alpha)-E_\gamma g(E_\alpha,E_\beta) \nonumber
\\ & \qquad -g(E_\alpha,[E_\beta,E_\gamma])+g(E_\beta,[E_\gamma,E_\alpha])+g(E_\gamma,[E_\alpha,E_\beta]). \label{koszuleq}
\end{align}

\begin{myrem} Note that the frame symbols $\Sigma_{\alpha \beta}^\gamma$ may not be symmetric in the $\alpha$ and $\beta$ indices since the (possibly non-zero) quantity $[E_\alpha,E_\beta]$ in the last term is anti-symmetric. This shows us that the frame symbols are not the same as the Christoffel symbols. \hfill $\blacksquare$
\end{myrem}

\begin{myprop} \label{framesymbol} Let $i,j,k \in \{1,2,3\}$. The frame symbols for the tangent bundle frame $\{\partial_z = E_0,E_1,E_2,E_3\}$ on $(M,g)$ are given by:

\begin{enumerate}
\item $\Sigma_{00}^0 =\frac{1}{2}g^{00}\partial_zg_{00}$, 
\item $\Sigma_{0i}^0=\Sigma_{i0}^0=\Sigma^i_{00}=0$,
\item $\Sigma_{j0}^i=\Sigma_{0j}^i=\frac{1}{2}g^{ii}\partial_z( \delta_i^j g_{ij})$,
\item $\Sigma_{ij}^0=-\frac{1}{2}g^{00}\partial_z(\delta_i^j g_{ij})$,
\item $\Sigma_{ij}^k=\hat{\Sigma}_{ij}^k=\epsilon_{ijk}g^{kk}(g_{ii}-g_{jj}-g_{kk})$ where $\hat{\Sigma}_{ij}^k$ is the frame symbol for the fibre metric $\hat{g}_z=a(z)^2 w^1 \otimes w^1 + b(z)^2 w^2 \otimes w^2 + c(z)^2 w^3 \otimes w^3$.  \hfill $\blacksquare$
\end{enumerate}

\end{myprop}

\begin{proof}
Equation (\ref{koszuleq}) can be simplified for certain combinations of $i,j$, and $k$. Note that $E_1,E_2,E_3 \in \mathscr{L}(TS^3)$ while $E_0 \in \mathscr{L}(TB)$, where $\mathscr{L}$ denotes the lift of the respective vector fields to $\Gamma(TM)$. Since $E_0$ and $E_i$ are in the horizontal and vertical lifts of the tangent bundle of product manifold, we have that $[E_0,E_i]=0$ for $i=1,2,3$. Furthermore, the frame is chosen to be orthogonal and $g(E_0,E_0)=\phi(z)^2=g_{00}$. Thus, if all three indices are $0$, then $2g_{00}\Sigma_{00}^0 =\partial_zg_{00}$. If exactly two of the indices in the frame symbol $\Sigma_{\alpha \beta}^\gamma$ are $0$, then the frame symbol vanishes identically. 

If exactly one of the indices $\alpha,\beta,$ or $\gamma$ is $0$, we have $\Sigma_{\alpha \beta}^\gamma= \Sigma_{\beta \alpha}^\gamma$ since the anti-symmetric term $g(E_\gamma,[E_\alpha,E_\beta])$ in (\ref{koszuleq}) vanishes. By symmetry, there are two cases: $\Sigma_{j0}^i$ or $\Sigma_{ij}^0$. In both cases, the last three terms of (\ref{koszuleq}) vanish since $\{E_1,E_2,E_3\}$ is chosen so that $[E_i,E_j]=-2\epsilon_{ijk} E_k$. Furthermore, since $E_0 \perp E_i$ for any $i=1,2,3$, only one of the remaining terms survives for both of these symbols. Respectively, we compute $2 g_{ii} \Sigma_{j0}^i= \partial_z(\delta_i^jg_{ij})$ and $2\Sigma_{ij}^0=-\partial_z (\delta_i^j g_{ij})$, which gives us the result.

Finally, for non-zero indices, if all three indices are the same, then clearly the frame symbols vanish. If exactly two of them are the same, there are three cases: $\Sigma_{ii}^j$, $\Sigma_{ji}^i$ or $\Sigma_{ij}^i$. In all cases, the first three terms of (\ref{koszuleq}) vanish since $a(z)$, $b(z)$, and $c(z)$ are constant in the fibres of each $z \in B$. This implies that the frame symbol $\Sigma_{ij}^k$ is the same as the frame symbol $\hat{\Sigma}_{ij}^k$ on the fibre $(S^3,\hat{g}_z)$. Furthermore, the last three terms also vanish since $[E_i,E_i]=0$ and $[E_i,E_j]=- 2\epsilon_{ijk} E_k \perp E_i$. 

Thus, the only (possibly) non-vanishing frame symbol of this form is when $i,j$, and $k$ are distinct. Substituting in all the Lie brackets of vector fields, we get $2g^{kk}\Sigma_{ij}^k=2(\epsilon_{jki} g_{ii}-\epsilon_{kij} g_{jj}-\epsilon_{ijk} g_{kk})$. By permuting the indices of the Levi-Civita symbol, we deduce the desired result.
\end{proof}

With these explicitly defined, we can calculate the Riemann curvature tensor components. Note here that we are calculating the tensor with respect to the frame $\{\partial_z=E_0,E_1,E_2,E_3\}$ instead of coordinate vector fields, so we denote $\text{Rm}_{\alpha\beta\gamma \delta}=g(\text{Rm}(E_\alpha,E_\beta)E_\gamma,E_\delta)$ for the indices $\alpha,\beta,\gamma, \delta \in \{0,1,2,3\}$.

\begin{myprop} \label{riemanntensors} Let $i,j,k \in \{1,2,3\}$ be distinct indices. The Riemann curvature tensor components in the frame $\{\partial_z = E_0,E_1,E_2,E_3\}$ on $(M,g)$ are given by: 
\begin{enumerate}
\item $\text{Rm}_{000i}=\text{Rm}_{0000}=0$,
\item $\text{Rm}_{i00j}=0$ and $\text{Rm}_{i00i}=\text{Rm}_{0ii0}=\frac{1}{4}(g^{00}g^{ii}\partial_zg_{00}\partial_zg_{ii}+g^{ii}(\partial_z g_{ii})^2-2\partial_{zz}g_{ii})$,
\item $\text{Rm}_{ijjk}=\text{Rm}_{ijj0}=0$ and $\text{Rm}_{ijji}=-\frac{1}{4}g^{00}\partial_zg_{jj} \partial_z g_{ii} + \hat{\text{Rm}}_{ijji}$ where $\hat{\text{Rm}}_{ijji} = -
g^{kk}(g_{kk}^2 -(g_{ii}-g_{jj})^2)-2(g_{kk}-g_{jj}-g_{ii})$ is the Riemann curvature tensor of the fibre metric $\hat{g}_z$. \hfill $\blacksquare$
\end{enumerate}

\end{myprop}
\begin{proof}
From the previous proposition, we note that the frame symbols are only dependent on the variable $z$. The first identity is clear. For the second, we compute:
\begin{align*}
\nabla^2E_0(E_i,E_0)-\nabla^2E_0(E_0,E_i) &=\nabla_{E_i}\nabla_{E_0}E_0-\nabla_{E_0} \nabla_{E_i} E_0 -\nabla_{[E_i,E_0]}E_0
\\ & = \nabla_{E_i}(\Sigma^0_{00}E_0)- \nabla_{E_0} (\Sigma_{i0}^i E_i) 
\\ &= \Sigma_{00}^0\Sigma_{i0}^iE_i-\partial_z \Sigma_{i0}^i E_i - \Sigma_{i0}^i \nabla_{E_0}E_i 
\\ &= \Sigma_{00}^0\Sigma_{i0}^iE_i-\partial_z \Sigma_{i0}^i E_i - \Sigma_{i0}^i \Sigma_{0i}^i E_i
\\ & = \frac{1}{4}g^{00}g^{ii}\partial_zg_{00}\partial_zg_{ii}E_i-\frac{1}{2}\partial_z(g^{ii}\partial_z g_{ii})E_i-\frac{1}{4}(g^{ii}\partial_z g_{ii})^2 E_i.
\end{align*}

By using the fact that $0=\partial_s(g^{ii}g_{ii})=g_{ii}\partial_sg^{ii}+g^{ii}\partial_sg_{ii}$, we can simplify this to:
$$\nabla^2E_0(E_i,E_0)-\nabla^2E_0(E_0,E_i)=\left(\frac{1}{4}g^{00}g^{ii}\partial_zg_{00}\partial_zg_{ii}+ \frac{1}{4}(g^{ii}\partial_zg_{ii})^2-\frac{1}{2}g^{ii}\partial_{zz}g_{ii} \right)E_i. $$

By taking the inner product of this quantity with $E_i$ and $E_j$ and using the symmetries of the Riemann curvature tensor, we obtain the second relation. For the third, we calculate:
\begin{align*}
& \nabla^2E_j(E_i,E_j)-\nabla^2E_j(E_j,E_i)
\\ &=\nabla_{E_i}\nabla_{E_j} E_j -\nabla_{E_j}\nabla_{E_i}E_j -\nabla_{[E_i,E_j]}E_j
\\ &= \nabla_{E_i}(\Sigma_{jj}^0E_0)-\nabla_{E_j}(\Sigma_{ij}^kE_k)+ 2\epsilon_{ijk}\nabla_{E_k}E_j
\\ &= \Sigma_{jj}^0\Sigma_{i0}^iE_i -\Sigma_{ij}^k \Sigma_{jk}^i E_i +\epsilon_{ijk}\Sigma_{kj}^i E_i
\\ &= -\frac{1}{4}g^{00}g^{ii} \partial_z g_{ii}\partial_z g_{jj}E_i- \epsilon_{ijk}\epsilon_{jki} g^{ii}g^{kk}(g_{ii}-g_{jj}-g_{kk})(g_{jj}-g_{kk}-g_{ii}) E_i
\\ & \qquad +2 \epsilon_{ijk}\epsilon_{kji}g^{ii}(g_{kk}-g_{jj}-g_{ii})E_i
\\ & =\left(-\frac{1}{4}g^{00}g^{ii}\partial_zg_{jj} \partial_z g_{ii}- g^{ii}g^{kk}(g_{kk}^2 -(g_{ii}-g_{jj})^2)-2g^{ii}(g_{kk}-g_{jj}-g_{ii}) \right) E_i.
\end{align*}

Again, by taking the inner product with $E_0, E_i$, and $E_k$, and using the symmetries of the Riemann curvature tensor, we extract the results. Note that throughout the calculation above, the last two terms do not involve the vector $E_0$ at all, hence they are the same as the Riemann curvature tensor $\hat{\text{Rm}}_{ijjk}$ of the fibre metric  $\hat{g}_z$.
\end{proof}

We can also calculate the Ricci tensor components and the sectional curvatures by taking the appropriate traces of the Riemann curvature tensor from above.

\begin{myprop} We have the following Ricci tensor components:
\begin{enumerate}
\item $\text{Ric}_{\alpha \beta}=0$ for $\alpha \neq \beta$,
\item $ \text{Ric}_{00}=\frac{1}{4} \displaystyle\sum_{i=1}^3 \left(g^{00}(g^{ii})^2\partial_zg_{00}\partial_zg_{ii}+(g^{ii}\partial_z g_{ii})^2-2g^{ii}\partial_{zz}g_{ii}) \right)$,
\item $ \text{Ric}_{ii}=\frac{1}{4}\displaystyle\big(g^{00}(g^{00}g^{ii}\partial_zg_{00}\partial_zg_{ii}+g^{ii}(\partial_z g_{ii})^2 - 2\partial_{zz}g_{ii})- \sum_{\substack{j=1 \\j\neq i \\ j \neq k}}^3 g^{jj}(g^{00}\partial_zg_{jj} \partial_z g_{ii}+4g^{kk}(g_{kk}^2 -(g_{ii}-g_{jj})^2)+8(g_{kk}-g_{jj}-g_{ii}))\big)$. \hfill $\blacksquare$
\end{enumerate}
\end{myprop}

\begin{myprop} \label{seccurve} For $i,j \in \{1,2,3\}$, the sectional curvatures of the metric $g$ are given by:
\begin{enumerate}
\item $K_{0i}= \frac{1}{4}g^{00}g^{ii}(g^{00}g^{ii}\partial_zg_{00}\partial_zg_{ii}+g^{ii}(\partial_z g_{ii})^2-2\partial_{zz}g_{ii})$,
\item $K_{ij}= -\frac{1}{4}
g^{00}g^{ii}g^{jj}\partial_z g_{ii} \partial_z g_{jj}+\hat{K}_{ij}$ where $\hat{K}_{ij}=-g^{ii}g^{jj}(
g^{kk}(g_{kk}^2 -(g_{ii}-g_{jj})^2)+(g_{kk}-g_{jj}-g_{ii}))$ for $k \neq i,j$ is the sectional curvature of the fibre metric  $\hat{g}_z$ spanned by $E_i$ and $E_j$. \hfill $\blacksquare$
\end{enumerate}
\end{myprop}

Recall that we have defined an arclength coordinate in the cylinder-to-sphere rule. The arclength element $ds$ is induced from $\phi(z)$ by defining the arclength coordinate $s(z)=\int_0^z \phi(w) \, dw$ and using this coordinate in place of $z$. In fact, this is a more natural geometric quantity as the formulation for the Ricci flow equation written in this coordinate, as we shall see later, is strongly parabolic. Therefore, we would not have to resort to DeTurck's Trick to formulate a system of parabolic PDEs.

 If we choose to calculate the symbols in the arclength coordinate frame (that is, $\{\partial_s=E_0,E_1,E_2,E_3\}$), we simply substitute $g^{00}=1$ in Propositions \ref{framesymbol}-\ref{seccurve} to get: 
 \begin{mycor} In the arclength coordinate frame, the Christoffel symbols can be written explicitly as:
\begin{enumerate}
\item $\Sigma_{00}^0 = \Sigma_{0i}^0=\Sigma_{i0}^0=\Sigma^i_{00}=0$,
\item $\Sigma_{j0}^i=\Sigma_{0j}^i=\frac{1}{2}g^{ii}\partial_s( \delta_i^j g_{ij})$,
\item $\Sigma_{ij}^0=-\frac{1}{2}\partial_s (\delta_i^j g_{ij})$,
\item $\Sigma_{ij}^k=\hat{\Sigma}_{ij}^k=\epsilon_{ijk}g^{kk}(g_{ii}-g_{jj}-g_{kk})$ where $\hat{\Sigma}_{ij}^k$ is the frame symbol for the fibre metric $\hat{g}_s=a(s)^2 \omega^1 \otimes \omega^1 + b(s)^2 \omega^2 \otimes \omega^2 + c(s)^2 \omega^3 \otimes \omega^3$. \hfill $\blacksquare$
\end{enumerate}
\end{mycor}
\begin{mycor}

In the arclength coordinate frame, the Riemmann curvature tensor components for  distinct indices $i,j,k \in \{1,2,3\}$ can be written explicitly as:
\begin{enumerate}
\item $\text{Rm}_{000i}=\text{Rm}_{0000}=0$,
\item $\text{Rm}_{i00j}=0$ and $\text{Rm}_{i00i}=\text{Rm}_{0ii0}=\frac{1}{4}(g^{ii}(\partial_s g_{ii})^2-2\partial_{ss}g_{ii})$,
\item $\text{Rm}_{ijjk}=\text{Rm}_{ijj0}=0$ and $\text{Rm}_{ijji}=-\frac{1}{4}\partial_sg_{jj} \partial_s g_{ii} + \hat{\text{Rm}}_{ijji}$ where $\hat{\text{Rm}}_{ijji} = -
g^{kk}(g_{kk}^2 -(g_{ii}-g_{jj})^2)-2(g_{kk}-g_{jj}-g_{ii})$ is the Riemann curvature tensor of the fibre metric $\hat{g}_s$.\hfill $\blacksquare$
\end{enumerate}
\end{mycor} 

\begin{mycor}
In the arclength coordinate frame, the Ricci curvature tensor components can be written explicitly as:
\begin{enumerate}
\item $\text{Ric}_{\alpha \beta}=0$ for $\alpha \neq \beta$,
\item $ \text{Ric}_{00}=\frac{1}{4} \displaystyle\sum_{i=1}^3 \left((g^{ii}\partial_s g_{ii})^2-2g^{ii}\partial_{ss}g_{ii}) \right)$,
\item $ \text{Ric}_{ii}=\frac{1}{4}\displaystyle\big(g^{ii}(\partial_z g_{ii})^2 - 2\partial_{zz}g_{ii}- \sum_{\substack{j=1 \\j\neq i \\ j \neq k}}^3 g^{jj}(\partial_zg_{jj} \partial_z g_{ii}+4g^{kk}(g_{kk}^2 -(g_{ii}-g_{jj})^2)+8(g_{kk}-g_{jj}-g_{ii}))\big)$.\hfill $\blacksquare$
\end{enumerate}
\end{mycor} 
\begin{mycor}
In the arclength coordinate frame, the sectional curvatures can be written explicitly as:
\begin{enumerate}
\item $K_{0i}= \frac{1}{4}((g^{ii}\partial_s g_{ii})^2-2g^{ii}\partial_{ss}g_{ii})$,
\item $K_{ij}= -\frac{1}{4}
g^{ii}g^{jj}\partial_s g_{ii} \partial_s g_{jj}+\hat{K}_{ij}$ where $\hat{K}_{ij}=-g^{ii}g^{jj}(
g^{kk}(g_{kk}^2 -(g_{ii}-g_{jj})^2)+(g_{kk}-g_{jj}-g_{ii}))$ for $k \neq i,j$ is the sectional curvature of the fibre metric  $\hat{g}_s$ spanned by $E_i$ and $E_j$. \hfill $\blacksquare$
\end{enumerate}

\end{mycor}

Explicitly, in the arclength coordinate frame $\{\partial_s = E_0,E_1,E_2,E_3\}$, we can calculate the Ricci tensor components of the metric in (\ref{ansatzmetric}) to get: 
\begin{align}
\text{Ric}_{00}&=-\left(\frac{a''}{a}+\frac{b''}{b}+\frac{c''}{c} \right), \label{ric1}
\\ \text{Ric}_{11}&=-aa''-aa'\left(\frac{b'}{b}+\frac{c'}{c} \right) + a^2(\hat{K}_{12}+\hat{K}_{13}),
\\ \text{Ric}_{22}&=-bb''-bb'\left(\frac{a'}{a}+\frac{c'}{c} \right) + b^2(\hat{K}_{12}+\hat{K}_{23}),
\\ \text{Ric}_{33}&=-cc''-cc'\left(\frac{a'}{a}+\frac{b'}{b} \right) + c^2(\hat{K}_{13}+\hat{K}_{23}),
\\ \text{Ric}_{\alpha \beta} &= 0 \quad \text{for} \quad \alpha \neq \beta, \label{ric5}
\end{align}
where:
\begin{align*}
\hat{K}_{12} &= \frac{(a^2-b^2)^2-3c^4}{(abc)^2}+\frac{2}{a^2}+\frac{2}{b^2},
\\ \hat{K}_{13} &=\frac{(a^2-c^2)^2-3b^4}{(abc)^2}+\frac{2}{a^2}+\frac{2}{c^2},
\\ \hat{K}_{23} &=\frac{(b^2-c^2)^2-3a^4}{(abc)^2}+\frac{2}{b^2}+\frac{2}{c^2},
\end{align*}
are the sectional curvature of the $\text{SU}(2)$ fibres. These were calculated in Proposition \ref{seccurve}. 

The $\ldots'$ denotes derivative with respect to the arclength variable $s(z,t)$  via the change of variable $ds = \phi(z,t) \, dz$. From now on, we denote derivatives with respect to the original variable $z$ and the new gauge $s(z,t)$ by $\dot{\ldots}$ and $\ldots'$ respectively. 

\begin{myrem} Note that by choosing this gauge, the new space variable $s$ depends on both the original space and time variables. Furthermore, from the relation $ds = \phi(z,t) dz$, we have the chain rule identity $\frac{\partial}{\partial s}=\frac{1}{\phi(z,t)}\frac{\partial}{\partial z}$. Thus, the derivatives in the $t$ direction and $s$ direction do not commute as the original variables $z$ and $t$ do. Therefore, we have the following commutator relation:
\begin{equation} \label{commutator}
[\partial_t,\partial_s]=-\frac{\partial_t \phi}{\phi} \, \partial_s.
\end{equation}

From now on, keeping this in mind, we suppress the dependence of the variable $s$ on the variables $z$ and $t$. \hfill $\blacksquare$
\end{myrem}

The sectional curvatures of the manifold $B \times S^3$ are given by:
\begin{align*}
K_{01}&=-\frac{a''}{a}
\\  K_{02} &=-\frac{b''}{b}, 
\\  K_{03} &=-\frac{c''}{c},
\\ K_{12} &=-\frac{a'b'}{ab}+\frac{(a^2-b^2)^2-3c^4}{(abc)^2}+\frac{2}{a^2}+\frac{2}{b^2}=-\frac{a'b'}{ab}+\hat{K}_{12},
\\ K_{13} &=-\frac{a'c'}{ac}+\frac{(a^2-c^2)^2-3b^4}{(abc)^2}+\frac{2}{a^2}+\frac{2}{c^2}=-\frac{a'c'}{ac}+\hat{K}_{13},
\\ K_{23} &=-\frac{b'c'}{bc}+\frac{(b^2-c^2)^2-3a^4}{(abc)^2}+\frac{2}{b^2}+\frac{2}{c^2}=-\frac{b'c'}{bc}+\hat{K}_{23}.
\end{align*}

Since the off-diagonal terms in the Ricci tensor vanish identically, the Ricci flow equation preserves the form of the initial metric; that is, for all time for which the solution exists, the metric $g(t)$ would be of the form (\ref{ansatzmetric}) for some positive functions $\phi,a,b$, and $c$ which are all functions of $z$ and $t$. Equating the first component in the Ricci flow equation gives us:
\begin{align} 
\partial_t (\phi^2)&=2\left(\frac{a''}{a} +\frac{b''}{b} + \frac{c''}{c} \right)\phi^2 \nonumber
\\ \Rightarrow \quad \partial_t(\log \phi)&=\frac{a''}{a}+\frac{b''}{b}+\frac{c''}{c}=-(K_{01}+K_{02}+K_{03}),\label{arclengthvalue}
\end{align}
and thus $[\partial_t,\partial_s]=(K_{01}+K_{02}+K_{03}) \, \partial_s$. 

The remaining equations are:
\begin{align*} 
\partial_t a &= a''+a'\left(\frac{b'}{b}+\frac{c'}{c} \right) - a(\hat{K}_{12}+\hat{K}_{13}),
\\ \partial_t b &= b''+b'\left(\frac{a'}{a}+\frac{c'}{c} \right) - b(\hat{K}_{12}+\hat{K}_{23}), 
\\ \partial_t c &= c''+c'\left(\frac{a'}{a}+\frac{b'}{b} \right) - c(\hat{K}_{13}+\hat{K}_{23}).
\end{align*}

In the arclength coordinate $s$, the Ricci flow is a semilinear parabolic system of equations:
\begin{align} \label{riccipde1}
\partial_t a &= a''+a'\left(\frac{b'}{b}+\frac{c'}{c} \right) - 2a\left(\frac{a^4-(b^2-c^2)^2}{(abc)^2}\right),
\\ \partial_t b &= b''+b'\left(\frac{a'}{a}+\frac{c'}{c} \right) - 2b\left(\frac{b^4-(a^2-c^2)^2}{(abc)^2}\right), \label{riccipde2}
\\ \partial_t c &= c''+c'\left(\frac{a'}{a}+\frac{b'}{b} \right) - 2c\left(\frac{c^4-(a^2-b^2)^2}{(abc)^2}\right). \label{riccipde3}
\end{align}

This system, along with the commutator relation $ [\partial_t,\partial_s]=-\left( \frac{a''}{a}+\frac{b''}{b}+\frac{c''}{c} \right) \partial_s$, will be the system of equations that we will analyse in the next few chapters. Furthermore, from equations (\ref{ric1})-(\ref{ric5}), we can calculate the scalar curvature $\text{S}$ for $g$, which will be useful later:
\begin{align} \text{S}&=2(K_{01}+K_{02}+K_{03}+K_{12}+K_{13}+K_{23}) \nonumber
\\&=2\left(-\frac{a''}{a}-\frac{b''}{b}-\frac{c''}{c}-\frac{a'b'}{ab}-\frac{a'c'}{ac}-\frac{b'c'}{bc}+\frac{2a^2b^2+2a^2c^2+2b^2c^2-a^4-b^4-c^4}{a^2b^2c^2}\right).\label{scal}
\end{align}

Under the Ricci flow, the scalar curvature evolves according to:
\begin{equation} \label{scalarevolution} \partial_t \text{S} = \Delta \text{S} + 2 |\text{Ric}|^2_g \geq \Delta \text{S} +\frac{2}{n}\text{S}^2. \end{equation}

%First, the quantities $a^2,b^2$, and $c^2$ evolve according to:
%\begin{align}
%\partial_t(a^2)&=2aa''+2aa'\left(\frac{b'}{b}+\frac{c'}{c} \right)-4\left(\frac{a^4-(b^2-c^2)^2}{b^2c^2} \right), \label{asquared}
%\\ \partial_t(b^2)&=2bb''+2bb'\left(\frac{a'}{a}+\frac{c'}{c} \right)-4\left(\frac{b^4-(a^2-c^2)^2}{a^2c^2} \right),
%\\ \partial_t(c^2)&=2cc''+2cc'\left(\frac{a'}{a}+\frac{b'}{b} \right)-4\left(\frac{c^4-(a^2-b^2)^2}{a^2b^2} \right). \label{csquared}
%\end{align}
To help us with the analysis, the evolution of the some derived quantities are also considered.
We first consider the quantity $\xi := a-b$. From equations (\ref{riccipde1}) and (\ref{riccipde2}), we can derive the evolution equation for $\xi$, which is given by:
$$\partial_t \xi =\xi''+\frac{c'}{c} \xi'+\left(\frac{a'b'}{ab}-\frac{2(a^4+b^4-c^4)}{a^2b^2c^2} -\frac{4(a^2+b^2-c^2)}{abc^2}-\frac{4}{c^2} \right)\xi. $$

Similarly, the quantities $\zeta:=b-c$ and $\chi:=a-c$ satisfy the equations:
\begin{align*}
\partial_t \zeta &= \zeta'' +\frac{a'}{a}\zeta'+\left(\frac{b'c'}{bc}-\frac{2(b^4+c^4-a^4)}{a^2b^2c^2} - \frac{4(b^2+c^2-a^2)}{a^2bc}-\frac{4}{a^2} \right)\zeta,
\\ \partial_t \chi &= \chi'' +\frac{b'}{b}\chi'+\left(\frac{a'c'}{ac}-\frac{2(a^4+c^4-b^4)}{a^2b^2c^2} - \frac{4(a^2+c^2-b^2)}{ab^2c}-\frac{4}{b^2} \right)\chi.
\end{align*}

In the vein the analysis in \cite{IKS}, we consider the quantities $\frac{\xi}{a}=\frac{a-b}{a}$ and $\frac{\xi}{b}=\frac{a-b}{b}$. These quantities can be thought of the measure for eccentricity of the fibre at each $z \in B$. These eccentricity quantities evolve according to the PDEs:
\begin{align}
\partial_t \left(\frac{\xi}{a} \right)&=\left(\frac{\xi}{a} \right)''+\left(2\frac{a'}{a}+\frac{c'}{c}\right)\left(\frac{\xi}{a} \right)'-4\left(\frac{1}{c^2}-\frac{1}{a^2}+\frac{b^2}{a^2c^2}+\frac{a^2+b^2-c^2}{abc^2}\right)\left(\frac{\xi}{a} \right), \label{ecc1}
\\ \partial_t \left(\frac{\xi}{b} \right)&=\left(\frac{\xi}{b} \right)''+\left(2\frac{b'}{b}+\frac{c'}{c}\right)\left(\frac{\xi}{b} \right)'-4\left(\frac{1}{c^2}-\frac{1}{b^2}+\frac{a^2}{b^2c^2}+\frac{a^2+b^2-c^2}{abc^2}\right)\left(\frac{\xi}{b} \right). \label{ecc2}
\end{align}

Furthermore, the other eccentricity quantities, defined by  $\frac{\zeta}{b}=\frac{b-c}{b}, \frac{\zeta}{c}=\frac{b-c}{c}, \frac{\chi}{a}=\frac{a-c}{a}$, and $\frac{\chi}{c}=\frac{a-c}{c}$ evolve according to the following PDEs:
\begin{align}
\partial_t \left(\frac{\zeta}{b} \right)&=\left(\frac{\zeta}{b} \right)''+\left(2\frac{b'}{b}+\frac{a'}{a}\right)\left(\frac{\zeta}{b} \right)'-4\left(\frac{1}{a^2}-\frac{1}{b^2}+\frac{c^2}{a^2b^2}+\frac{b^2+c^2-a^2}{a^2bc}\right)\left(\frac{\zeta}{b} \right), \label{ecc3}
\\ \partial_t \left(\frac{\zeta}{c} \right)&=\left(\frac{\zeta}{c} \right)''+ \left(2\frac{c'}{c}+\frac{a'}{a}\right)\left(\frac{\zeta}{c} \right)'-4\left(\frac{1}{a^2}-\frac{1}{c^2}+\frac{b^2}{a^2c^2}+\frac{b^2+c^2-a^2}{a^2bc}\right)\left(\frac{\zeta}{c} \right), \label{ecc4}
\\ \partial_t \left(\frac{\chi}{a} \right)&=\left(\frac{\chi}{a} \right)''+\left(2\frac{a'}{a}+\frac{b'}{b}\right)\left(\frac{\chi}{a} \right)'-4\left(\frac{1}{b^2}-\frac{1}{a^2}+\frac{c^2}{a^2b^2}+\frac{a^2+c^2-b^2}{ab^2c}\right)\left(\frac{\chi}{a} \right), \label{ecc5}
\\ \partial_t \left(\frac{\chi}{c} \right)&=\left(\frac{\chi}{c} \right)''+\left(2\frac{a'}{a}+\frac{c'}{c}\right)\left(\frac{\chi}{c} \right)'-4\left(\frac{1}{b^2}-\frac{1}{c^2}+\frac{a^2}{b^2c^2}+\frac{a^2+c^2-b^2}{ab^2c}\right)\left(\frac{\chi}{c} \right). \label{ecc6}
\end{align}

%In addition, the quantity $\kappa=\frac{a-b}{c}$ evolves by:
%$$\partial_t \kappa=\kappa''+3\frac{c'}{c}\kappa'+ \left( \left(\frac{c'}{c}-\frac{a'}{a}\right)\left(\frac{c'}{c}-\frac{b'}{b} \right)-\frac{a^4+b^4-c^4}{a^2b^2c^2}-\frac{a^2+b^2-c^2}{abc^2} \right)\kappa.$$

%Finally, the quantity $\nu = \frac{c^2 - a^2 - b^2}{2}$ evolves according to the equation:
%\begin{equation}
%\partial_t \nu = \nu''+\left(\frac{a'}{a}+\frac{b'}{b}+\frac{c'}{c} \right)\nu'-\frac{(a^2+b^2+c^2)^2}{a^2b^2c^2}\nu-\frac{4(a^2+b^2)}{c^2}+(a')^2+(b')^2-(c')^2.
%\end{equation}

To calculate the evolution of the curvatures, we need to find the evolution equations of the first and second derivatives of the metric components. For the first derivatives, by using the commutator relation (\ref{commutator}) to swap the order of the time and space derivatives $\partial_t$ and $\partial_s$, we derive the following evolution equations:
\begin{align}
\partial_t(a')&=(\partial_t a)'+[\partial_t,\partial_s]a'=(\partial_t a)'-\left(\frac{a''}{a}+\frac{b''}{b}+\frac{c''}{c} \right)a' \nonumber
\\ &= a''' +a''\left(\frac{b'}{b}+\frac{c'}{c}-\frac{a'}{a} \right)-a'\left(\frac{(b')^2}{b^2}+\frac{(c')^2}{c^2}+\frac{6a^4+2(b^2-c^2)^2}{(abc)^2} \right) \nonumber
\\ & \qquad +4 \left(\frac{a^3b'}{b^3c^2}+\frac{a^3c'}{b^2c^3}-\frac{c^2b'}{ab^3}-\frac{b^2c'}{ac^3}+\frac{cc'}{ab^2}+\frac{bb'}{ac^2}\right), \label{aprime}
\\ \partial_t(b')&= b''' +b''\left(\frac{a'}{a}+\frac{c'}{c}-\frac{b'}{b} \right)-b'\left(\frac{(a')^2}{a^2}+\frac{(c')^2}{c^2}+\frac{6b^4+2(a^2-c^2)^2}{(abc)^2} \right) \nonumber
\\ & \qquad +4\left(\frac{b^3a'}{a^3c^2}+\frac{b^3c'}{a^2c^3}-\frac{c^2a'}{ba^3}-\frac{a^2c'}{bc^3}+\frac{cc'}{ba^2}+\frac{aa'}{bc^2}\right), \label{bprime}
\\ \partial_t(c')&= c''' +c''\left(\frac{a'}{a}+\frac{b'}{b}-\frac{c'}{c} \right)-c'\left(\frac{(a')^2}{a^2}+\frac{(b')^2}{b^2}+\frac{6c^4+2(a^2-b^2)^2}{(abc)^2} \right) \nonumber
\\ & \qquad +4\left(\frac{c^3a'}{a^3b^2}+\frac{c^3b'}{a^2b^3}-\frac{b^2a'}{ca^3}-\frac{a^2b'}{cb^3}+\frac{aa'}{cb^2}+\frac{bb'}{ca^2}\right), \label{cprime}
\end{align}

Furthermore, we have:
\begin{align}
\partial_t \left(\frac{(a')^2}{a^2} \right)&=\Delta_g \left(\frac{(a')^2}{a^2} \right)-2\frac{(a')^2}{a^2}\left(2\frac{(a')^2}{a^2}+ \frac{(b')^2}{b^2}+\frac{(c')^2}{c^2}+\frac{4(a^4+(b^2-c^2)^2)}{a^2b^2c^2}\right) \nonumber
\\ &\qquad +8\frac{a'}{a^2}\left(\frac{a^3b'}{b^3c^2}+\frac{a^3c'}{b^2c^3}-\frac{c^2b'}{ab^3}-\frac{b^2c'}{ac^3}+\frac{cc'}{ab^2}+\frac{bb'}{ac^2}\right)-2K_{01}^2 -4\frac{(a')^2}{a^2}K_{01},\label{aprimeasquared}
\\ \partial_t \left(\frac{(b')^2}{b^2} \right)&=\Delta_g \left(\frac{(b')^2}{b^2} \right)-2\frac{(b')^2}{b^2}\left(2\frac{(b')^2}{b^2}+ \frac{(a')^2}{a^2}+\frac{(c')^2}{c^2}+\frac{4(b^4+(a^2-c^2)^2)}{a^2b^2c^2}\right) \nonumber
\\ &\qquad +8\frac{b'}{b^2}\left(\frac{b^3a'}{a^3c^2}+\frac{b^3c'}{a^2c^3}-\frac{c^2a'}{ba^3}-\frac{a^2c'}{bc^3}+\frac{cc'}{ba^2}+\frac{aa'}{bc^2}\right)-2K_{02}^2 -4\frac{(b')^2}{b^2}K_{02}, \label{bprimebsquared}
\\ \partial_t \left(\frac{(c')^2}{c^2} \right)&=\Delta_g \left(\frac{(c')^2}{c^2} \right)-2\frac{(c')^2}{c^2}\left(2\frac{(c')^2}{c^2}+ \frac{(a')^2}{a^2}+\frac{(b')^2}{b^2}+\frac{4(c^4+(a^2-b^2)^2)}{a^2b^2c^2}\right) \nonumber
\\ &\qquad +8\frac{c'}{c^2}\left(\frac{c^3a'}{a^3b^2}+\frac{c^3b'}{a^2b^3}-\frac{b^2a'}{ca^3}-\frac{a^2b'}{cb^3}+\frac{aa'}{cb^2}+\frac{bb'}{ca^2}\right)-2K_{03}^2 -4\frac{(c')^2}{c^2}K_{03}, \label{cprimecsquared}
\end{align}
where $\Delta_g$ is the Laplacian operator, which is defined as: 
$$\Delta_g f= f''+\left(\frac{a'}{a}+\frac{b'}{b}+\frac{c'}{c} \right)f',$$ for any $f \in C^2(I)$. As for the second derivatives, we can calculate:
\begin{align*}
\partial_t(a'')&=a''''+a'''\left(\frac{b'}{b}+\frac{c'}{c}-\frac{a'}{a} \right) +a''\left(\frac{(a')^2}{a^2}-2\frac{(b')^2}{b^2}-2\frac{(c')^2}{c^2}-2\frac{a''}{a}- \frac{6a^4+2(b^2-c^2)^2}{(abc)^2}\right)
\\ & \qquad -2a'\bigg(\frac{b'b''}{b^2}-\frac{(b')^3}{b^3}+\frac{c'c''}{c^2}-\frac{(c')^3}{c^3}+\frac{6aa'}{b^2c^2}+\frac{4bb'}{a^2c^2}+\frac{4cc'}{a^2b^2}-\frac{2a'c^2}{a^3b^2}-\frac{2a'b^2}{a^3c^2}+\frac{4a'}{a^3} 
\\ &\qquad -\frac{12a^2b'}{b^3c^2}-\frac{12a^2c'}{b^2c^3}-\frac{4b^2c'}{a^2c^3}-\frac{4c^2b'}{a^2b^3}\bigg)+4a^3\left(\frac{b''}{b^3c^2}-\frac{3(b')^2}{b^4c^2}-\frac{4b'c'}{b^3c^3}+\frac{c''}{b^2c^3}-\frac{3(c')^2}{b^2c^4} \right) 
\\ &\qquad +4\left(\frac{(c')^2+cc''}{ab^2}+\frac{(b')^2+bb''}{ac^2}+\frac{3(b')^2c^2}{ab^4}+\frac{3b^2(c')^2}{ac^4}-\frac{4cc'b'+c^2b''}{ab^3}-\frac{4bb'c'+b^2c''}{ac^3}\right),
\\ \partial_t(b'')&=b''''+b'''\left(\frac{a'}{a}+\frac{c'}{c}-\frac{b'}{b} \right) +b''\left(\frac{(b')^2}{b^2}-2\frac{(a')^2}{a^2}-2\frac{(c')^2}{c^2}-2\frac{b''}{b}- \frac{6b^4+2(a^2-c^2)^2}{(abc)^2}\right)
\\ & \qquad -2b'\bigg(\frac{a'a''}{a^2}-\frac{(a')^3}{a^3}+\frac{c'c''}{c^2}-\frac{(c')^3}{c^3}+\frac{6bb'}{a^2c^2}+\frac{4aa'}{b^2c^2}+\frac{4cc'}{a^2b^2}-\frac{2b'c^2}{b^3a^2}-\frac{2b'a^2}{b^3c^2}+\frac{4b'}{b^3} 
\\ & \qquad -\frac{12b^2a'}{a^3c^2}-\frac{12b^2c'}{a^2c^3}-\frac{4a^2c'}{b^2c^3}-\frac{4c^2a'}{b^2a^3}\bigg)+4b^3\left(\frac{a''}{a^3c^2}-\frac{3(a')^2}{a^4c^2}-\frac{4a'c'}{a^3c^3}+\frac{c''}{a^2c^3}-\frac{3(c')^2}{a^2c^4} \right) 
\\ & \qquad +4\left(\frac{(c')^2+cc''}{ba^2}+\frac{(a')^2+aa''}{bc^2}+\frac{3(a')^2c^2}{ba^4}+\frac{3a^2(c')^2}{bc^4}-\frac{4cc'a'+c^2a''}{ba^3}-\frac{4aa'c'+a^2c''}{bc^3}\right),
\\ \partial_t(c'')&=c''''+c'''\left(\frac{a'}{a}+\frac{b'}{b}-\frac{c'}{c} \right) +c''\left(\frac{(c')^2}{c^2}-2\frac{(a')^2}{a^2}-2\frac{(b')^2}{b^2}-2\frac{c''}{c}- \frac{6c^4+2(a^2-b^2)^2}{(abc)^2}\right)
\\ & \qquad -2c'\bigg(\frac{a'a''}{a^2}-\frac{(a')^3}{a^3}+\frac{b'b''}{b^2}-\frac{(b')^3}{b^3}+\frac{6cc'}{a^2b^2}+\frac{4aa'}{b^2c^2}+\frac{4bb'}{a^2b^2}-\frac{2c'b^2}{c^3a^2}-\frac{2c'a^2}{c^3b^2}+\frac{4c'}{c^3} 
\\ & \qquad -\frac{12c^2a'}{a^3b^2}-\frac{12c^2b'}{a^2b^3}-\frac{4b^2a'}{c^2a^3}-\frac{4a^2b'}{c^2b^3}\bigg)+4c^3\left(\frac{a''}{a^3b^2}-\frac{3(a')^2}{a^4b^2}-\frac{4a'b'}{a^3b^3}+\frac{b''}{a^2b^3}-\frac{3(b')^2}{a^2b^4} \right) 
\\ &\qquad +4\left(\frac{(b')^2+bb''}{ca^2}+\frac{(a')^2+aa''}{cb^2}+\frac{3a^2(b')^2}{cb^4}+\frac{3(a')^2b^2}{ca^4}-\frac{4aa'b'+a^2b''}{cb^3}-\frac{4bb'a'+b^2a''}{ca^3}\right).
\end{align*}

Thus, from the above equations, the sectional curvatures $K_{01},K_{02}$, and $K_{03}$ evolve according to the following PDEs:
\begin{align}
\partial_t K_{01}&=\partial_t\left(-\frac{a''}{a} \right)=-\frac{\partial_t(a'')}{a}+\frac{a''\partial_t a}{a^2} \nonumber
\\&=\Delta_g K_{01}+2K_{01}^2-2K_{01}\left( \frac{(b')^2}{b^2}+\frac{(c')^2}{c^2}+\frac{2a^4+2(b^2-c^2)^2}{(abc)^2}\right) \nonumber
\\ &\qquad +2K_{02}\left(\frac{2a^2}{b^2c^2}+\frac{2b^2}{a^2c^2}-\frac{2c^2}{a^2b^2}-\frac{a'b'}{ab} \right)+2K_{03}\left(\frac{2a^2}{b^2c^2}+\frac{2c^2}{a^2b^2}-\frac{2b^2}{a^2c^2}-\frac{a'c'}{ac} \right) \nonumber
\\ & \qquad +2\frac{a'}{a}\bigg(-\frac{(b')^3}{b^3}-\frac{(c')^3}{c^3}+\frac{6aa'}{b^2c^2}+\frac{4bb'}{a^2c^2}+\frac{4cc'}{a^2b^2}-\frac{2a'c^2}{a^3b^2}-\frac{2a'b^2}{a^3c^2}+\frac{4a'}{a^3}  \nonumber
\\ &\qquad -\frac{12a^2b'}{b^3c^2}-\frac{12a^2c'}{b^2c^3}-\frac{4b^2c'}{a^2c^3}-\frac{4c^2b'}{a^2b^3}\bigg)+4a^2\left(\frac{3(b')^2}{b^4c^2}+\frac{4b'c'}{b^3c^3}+\frac{3(c')^2}{b^2c^4} \right) \nonumber
\\ &\qquad -\frac{4}{a}\left(\frac{(c')^2}{ab^2}+\frac{(b')^2}{ac^2}+\frac{3(b')^2c^2}{ab^4}+\frac{3b^2(c')^2}{ac^4}-\frac{4cc'b'}{ab^3}-\frac{4bb'c'}{ac^3}\right), \label{K01eq}
\\ \partial_t K_{02}&=\Delta_g K_{02}+2K_{02}^2-2K_{02}\left( \frac{(a')^2}{a^2}+\frac{(c')^2}{c^2}+\frac{2b^4+2(a^2-c^2)^2}{(abc)^2}\right) \nonumber
\\ &\qquad +2K_{01}\left(\frac{2a^2}{b^2c^2}+\frac{2b^2}{a^2c^2}-\frac{2c^2}{a^2b^2}-\frac{a'b'}{ab} \right)+2K_{03}\left(\frac{2b^2}{a^2c^2}+\frac{2c^2}{a^2b^2}-\frac{2a^2}{b^2c^2}-\frac{b'c'}{bc} \right) \nonumber
\\ & \qquad +2\frac{b'}{b}\bigg(-\frac{(a')^3}{a^3}-\frac{(c')^3}{c^3}+\frac{6bb'}{a^2c^2}+\frac{4aa'}{b^2c^2}+\frac{4cc'}{a^2b^2}-\frac{2b'c^2}{b^3a^2}-\frac{2b'a^2}{b^3c^2}+\frac{4b'}{b^3}  \nonumber
\\ & \qquad -\frac{12b^2a'}{a^3c^2}-\frac{12b^2c'}{a^2c^3}-\frac{4a^2c'}{b^2c^3}-\frac{4c^2a'}{b^2a^3}\bigg)+4b^2\left(\frac{3(a')^2}{a^4c^2}+\frac{4a'c'}{a^3c^3}+\frac{3(c')^2}{a^2c^4} \right) \nonumber
\\ & \qquad -\frac{4}{b}\left(\frac{(c')^2}{ba^2}+\frac{(a')^2}{bc^2}+\frac{3(a')^2c^2}{ba^4}+\frac{3a^2(c')^2}{bc^4}-\frac{4cc'a'}{ba^3}-\frac{4aa'c'}{bc^3}\right), \label{K02eq}
\\ \partial_t K_{03}&=\Delta_g K_{03}+2K_{03}^2-2K_{03}\left( \frac{(a')^2}{a^2}+\frac{(b')^2}{b^2}+\frac{2c^4+2(a^2-b^2)^2}{(abc)^2}\right) \nonumber
\\ &\qquad +2K_{01}\left(\frac{2a^2}{b^2c^2}+\frac{2c^2}{a^2b^2}-\frac{2b^2}{a^2c^2}-\frac{a'c'}{ac} \right)+2K_{02}\left(\frac{2b^2}{a^2c^2}+\frac{2c^2}{a^2b^2}-\frac{2a^2}{b^2c^2}-\frac{b'c'}{bc} \right) \nonumber
\\ & \qquad +2\frac{c'}{c}\bigg(-\frac{(a')^3}{a^3}-\frac{(b')^3}{b^3}+\frac{6cc'}{a^2b^2}+\frac{4aa'}{b^2c^2}+\frac{4bb'}{a^2b^2}-\frac{2c'b^2}{c^3a^2}-\frac{2c'a^2}{c^3b^2}+\frac{4c'}{c^3} \nonumber
\\ & \qquad -\frac{12c^2a'}{a^3b^2}-\frac{12c^2b'}{a^2b^3}-\frac{4b^2a'}{c^2a^3}-\frac{4a^2b'}{c^2b^3}\bigg)+4c^2\left(\frac{3(a')^2}{a^4b^2}+\frac{4a'b'}{a^3b^3}+\frac{3(b')^2}{a^2b^4} \right) \nonumber
\\ &\qquad -\frac{4}{c}\left(\frac{(b')^2}{ca^2}+\frac{(a')^2}{cb^2}+\frac{3a^2(b')^2}{cb^4}+\frac{3(a')^2b^2}{ca^4}-\frac{4aa'b'}{cb^3}-\frac{4bb'a'}{ca^3}\right). \label{K03eq}
\end{align}

\section{Maximum Principles}
All of the equations computed in the previous section are of parabolic type. In order to study the behaviour of the metric components and the various other quantities computed earlier under the Ricci flow, we require some tools from parabolic theory, namely the maximum principles.  Let $I$ be a bounded open interval of $\mathbb{R}$ and define the initial-boundary value problem for a linear parabolic equation:
\begin{align} 
\mathcal{P}(u) = \partial_t u  - u''-f(x,t) u' -g(x,t)u&=0 \quad \text{on} \quad I\times [0,T) \subset \mathbb{R} \times \mathbb{R}_+, \label{pareqn1}
\\ u|_{I \times \{0\}}(x)&=v(x), \label{pareqn2}
\\ u|_{\partial I \times [0,T)}&=w(x,t). \label{pareqn3}
\end{align}

First, we define the parabolic boundary for the domain of the solution as follows:
\begin{mydef}[Parabolic domain, closure, interior, and boundary]
Suppose that the solution of the parabolic problem (\ref{pareqn1})-(\ref{pareqn3}) exists up to a time $T >0$. For a given $\tau \leq T$, the parabolic domain $I_\tau$ for the equation is given by the set $I \times [0,\tau)$. The closure and interior of $I_
\tau$ are defined as $\bar{I}_\tau=\bar{I} \times [0,\tau]$ and $\mathring{I}_\tau=\mathring{I} \times (0,\tau)$ respectively. Furthermore, the parabolic boundary of $I_\tau$   is given by the set $\mathscr{P} I_\tau = I \times \{0\} \cup \partial I \times [0,\tau)$. \hfill $\blacksquare$
 \end{mydef}

We now state the parabolic maximum principles that we are going to employ in the proofs of the results in this work.

% The first one is the classical maximum principle which bounds the solution of a parabolic problem by its values on the parabolic boundary whereas the second one gives a stronger bound obtained by studying the reaction term in the equation.

\begin{mythm}[Maximum Principle I]\label{maximumprinciple1} \cite{ELC} Suppose that $I$ is a bounded and connected open domain in $\mathbb{R}$. Let $\mathcal{P}(u)=\partial_t u  - u''-f(x,t) u' -g(x,t)u$ be a parabolic operator such that $f,g: I \times \mathbb{R}_+ \rightarrow \mathbb{R}$ are some functions with $g(x,t) \leq 0$ in $I_T$ where $T>0$. Suppose that $u \in C^{2,1}(I_T) \cap C^{0,0}(\bar{I}_T)$. Then:
\begin{enumerate}
\item If $u$ satisfies the parabolic inequality $\mathcal{P}(u) \leq 0$, then we have $\displaystyle \max_{I_T}(u)=\max_{\mathscr{P}I_T}(u^+)$ where $u^+=\max(u,0)$.
\item If $u$ satisfies the parabolic inequality $\mathcal{P}(u) \geq 0$, then we have $\displaystyle \min_{I_T}(u)\geq -\min_{\mathscr{P}I_T}(u^-)$ where $u^-=-\min(u,0)$.
\item In particular, if $\mathcal{P}(u)=0$, then we have $\displaystyle \max_{I_T}(|u|)=\max_{\mathscr{P}I_T}(|u|)$ and consequently $\displaystyle |u| \leq \max_{\mathscr{P}I_T}(|u|)$. \hfill $\blacksquare$
\end{enumerate}  
\end{mythm}

\begin{mythm}[Maximum Principle II] \cite{TP} \label{maximumprinciple2} Suppose that $I$ is a bounded open interval in $\mathbb{R}$ and $u \in C^{2,1}(I_T) \cap C^{0,0}(\bar{I}_T)$ satisfies the parabolic inequality: 
\begin{equation*}
\partial_t u  \leq u''+f(x,t) u' +g(u,t), 
\end{equation*}
 for some functions $f: I \times \mathbb{R}_+ \rightarrow \mathbb{R}$ and $g: \mathbb{R} \times \mathbb{R}_+ \rightarrow \mathbb{R}$. Assume that $g$ is locally Lipschitz in the $u$ variable and $u \leq C$ on $\mathscr{P}I_T$ where $T >0$ is the maximal time for which the solution exists. If $v$ is the solution of the associated ODE:
 \begin{align} \frac{dv}{dt} &= g(v,t),  \nonumber
\\ v(0)&=C, \nonumber \end{align}
then we have $u(x,t) \leq v(t)$ for all $(x,t) \in I_T$ for which $v$ exists. \hfill $\blacksquare$
\end{mythm}

%Note that in this chapter, our base manifold $I$ is the closed manifold $S^1$ and hence it has no boundary. Therefore, the parabolic boundary $\mathscr{P}S^1_T$ is simply $S^1 \times \{0\}$. Thus, the analysis is a lot simpler here than it is for the cases considered in the next chapters.

If we have non-positive data on the parabolic boundary, then we have a stronger result which does not require any sign conditions on the reaction term $g(x,t)$ as in Theorem \ref{maximumprinciple1}. A classical result \cite{PW} simply requires this reaction term to be uniformly bounded. In fact, the uniformly bounded condition can be weakened to locally bounded since the analysis are all done locally. For our purposes, we provide the proof for a simpler case for which we have vanishing boundary data and the coefficients are allowed to blow up to infinity at the boundaries. 

\begin{mythm}  \label{maximumprinciple3} Suppose that $I \subset \mathbb{R}$ is a bounded open interval and $u \in C^{2,1}(I_T) \cap C^{0,0}(\bar{I}_T)$ is a solution to the parabolic inequality $\mathcal{P}(u)=\partial_t u  - u''-f(x,t) u' -g(x,t)u \leq 0$ where $f,g: \mathring{I}_T \rightarrow \mathbb{R}$ are continuous functions in $I_T$. If $u \leq 0$ on $\mathscr{P}I_T$, then $u \leq 0$ in $I_T$. \hfill $\blacksquare$
\end{mythm}

\end{document}